\documentclass[12pt,twoside,leqno]{article}
\usepackage[T1]{fontenc}
\usepackage{amsmath,amsthm}
\usepackage{amssymb,latexsym}
\usepackage{enumerate}

\theoremstyle{plain}
\newtheorem{thm}{Theorem}[section]
\newtheorem{lem}[thm]{Lemma}
\newtheorem{prop}[thm]{Proposition}
\newtheorem{cor}[thm]{Corollary}

\newtheorem{f}[thm]{Fact}
\theoremstyle{definition}
\newtheorem{exam}[thm]{Example}
\newtheorem{defi}[thm]{Definition}
\theoremstyle{remark}
\newtheorem{rem}[thm]{Remark}
\newtheorem{q}[thm]{Question}
\newtheorem{p}[thm]{Problem}

\newcommand{\mb}{\mathbb}
\newcommand{\mc}{\mathcal}
\newcommand{\gts}{\mathbf{GTS}}

\newcommand{\vn}{\varnothing}

\newcommand{\cl}{\mathrm{cl}}

\begin{document}

\title{Compactness and compactifications in generalized topology}
\author{Artur Pi\k{e}kosz \\
Institute of Mathematics\\
Cracow University of Technology\\
Warszawska 24, \\
31-155 Cracow, Poland\\
E-mail: pupiekos@cyfronet.pl\\
and\\
Eliza Wajch\\
Institute of Mathematics and Physics\\
Siedlce University of Natural Sciences and Humanities\\
3 Maja 54,\\
08-110 Siedlce, Poland\\
E-mail: eliza.wajch@wp.pl }
\date{}
\maketitle

\begin{abstract}
A generalized topology in a set $X$ is a collection $\text{Cov}_X$ of
families of subsets of $X$ such that the triple $(X, \bigcup{\text{Cov}_X},
\text{Cov}_X)$ is a generalized topological space in the sense of Delfs and
Knebusch. In this work, notions of topological and admissible compactness of
generalized topologies are introduced to begin and investigate a theory of
compactifications, in particular, of Wallman type in the category of weakly normal
generalized topological spaces. Among other facts, we prove in \textbf{ZF}
that the ultrafilter theorem (in abbreviation \textbf{UFT}) holds if and
only if all Wallman extensions of every weakly normal generalized
topological space are compact. In consequence, we develop the theory of
compactifications in \textbf{ZF}+\textbf{UFT} when it is not necessary to
use \textbf{AC}, while \textbf{ZF} is not enough.
\end{abstract}

\renewcommand{\thefootnote}{}



\renewcommand{\thefootnote}{\arabic{footnote}} \setcounter{footnote}{0}

\section{Introduction}

Classical general topology, in particular the theory of compactifications,
as well as most parts of current mathematics, have been created under the
basic assumption of all axioms of $ZFC$ and of the consistency of $ZFC$
although some mathematicians seemed to use the assumptions rather
unconsciously. Relatively few authors investigate in deep the role of the
axiom of choice denoted by \textbf{AC} and, if this is possible, analyse
whether the proofs of known theorems in $ZFC$ in which at least countable
choice has been involved can be modified to proofs in $ZF$ or in $ZF$ plus a
weaker assumption than \textbf{AC} (cf. e.g. \cite{Her}, \cite{Jech1} and 
\cite{Jech2}). To avoid careless usage of \textbf{AC}, we are going to
work in the spirit of Horst Herrlich shown partly in \cite{Her}. In this
paper, both notation and terminology for $ZFC$ are taken from \cite{Her}, 
\cite{Jech1}, \cite{Jech2}. \cite{Ku1},\cite{Ku2}. To make progress faster
and come to the essence of the paper, let us assume the system \textbf{ZF}
which, in informal words, is a convenient interpretation of more or less 
$ZF$+[Axioms of Logic] from \cite{Ku2} with the exception that, in much the same way, as in 
\textbf{NBG} or \textbf{MK} and contrary to \cite{Ku1}-\cite{Ku2}, we
modify $ZF$ in such a way that proper classes exist in \textbf{ZF}. We
neither give formulations of axioms of \textbf{ZF} nor describe formal
rules of deduction here but we hope that readers' intuition is good enough
to understand the content of the paper. We apply rather commonly used by
mathematicians informal law that, for every non-void finite collection of
pairwise disjoint non-void sets, one has a choice function of this
collection.

Our terminology concerning usual general topology and category theory, if
not introduced or modified below, is standard and can be found in \cite{En}, 
\cite{PW}, \cite{Ch}, \cite{Her}, \cite{AHS}. We use conventions of \cite{DK},
 \cite{Pie1} and \cite{Pie2} for generalized topological spaces.

\begin{defi}
A \textbf{generalized topology} in a set $X$ is a collection $\text{Cov}_X$ of
families of subsets of $X$ such that the triple $(X, \bigcup{\text{Cov}_X},
\text{Cov}_X)$ is a generalized topological space (abbr. gts) in the sense
of Delfs and Knebusch (cf. \cite{DK} and Definitions 2.2.1--2.2.2 with Remark 2.2.3
in \cite{Pie1}).
\end{defi}

Let $\text{Cov}_X$ be a generalized topology in $X$.  Only sets
from the collection $\text{Op}_X=\bigcup\text{Cov}_X$ are called \textbf{open}
 in the gts $(X, \text{Op}_X,\text{Cov}_X)$, while sets from the collection 
$\text{Cl}_X=\{A\subseteq X: X\setminus A\in\text{Op}_X\}$ are called 
\textbf{closed} in the gts $(X, \bigcup\text{Cov}_X,\text{Cov}_X)$. The gts 
$(X, \bigcup\text{Cov}_X,\text{Cov}_X)$ can be denoted briefly by $X$ or by 
$(X, \text{Cov}_X)$ if this is not misleading. Families that belong to 
$\text{Cov}_X$ are called \textbf{admissible coverings} or \textbf{admissible open families} in $X$. One can guess that \textbf{admissible closed families} in $X$ are  families $\mathcal{H}\subseteq \mathcal{P}(X)$ such that $\{ X\setminus H: H\in\mathcal{H}\}$ are admissible open families. 
If $A\subseteq\mathcal{P}(X)$, then $\tau (A)$ denotes the topology in the standard sense generated by $A$.
The sets from the topology $\tau(\text{Op}_X)$ are called \textbf{weakly open}
 in the gts $(X, \text{Cov}_X)$. The sets closed in the topological space 
$(X, \tau(\text{Op}_X))$ are called \textbf{weakly closed} in the gts $(X, 
\text{Cov}_X)$. If $\mathcal{A}\subseteq\mathcal{P}^{2}(X)$, then $\langle 
\mathcal{A}\rangle$ is the smallest (with respect to inclusion) generalized
topology in $X$ that contains $\mathcal{A}$. When it is necessary to tell
exactly that we generate from $\mathcal{A}$ admissible coverings in $X$, we
will use the notation $\langle \mathcal{A}\rangle_X$. If $Y$ is a proper subset of $X$
and $\mathcal{A}\subseteq\mathcal{P}^{2}(Y)$,
the collections $\langle \mathcal{A}\rangle_X$ and $\langle\mathcal{A}
\rangle_Y$ are different. When $Y\subseteq X$, the pair $(Y, \langle
\text{Cov}_X\cap_{2}Y\rangle_Y)$ is called a \textbf{subspace} of the gts $(X, 
\text{Cov}_X)$. A collection $\mathcal{U}\subseteq\mathcal{P}(X)$ is called 
\textbf{essentially finite} if there is a finite subcollection $\mathcal{V}$
of $\mathcal{U}$ such that $\bigcup\mathcal{U}=\bigcup\mathcal{V}$. Let us
notice that $\langle \{\text{Op}_X \} \rangle$ is the collection $\text{EssFin} 
(\text{Op}_X)$ of all essentially finite families of open sets in the gts $
(X, \text{Cov}_X)$. The gts $(X, \text{Op}_X, \text{EssFin}(\text{Op}_X))$
is the \textbf{smallification} of $(X, \text{Op}_X,\text{Cov}_X)$. The gts $
(X, \text{Op}_X,\text{Cov}_X)$ is called \textbf{small} if $\text{Cov}_X= 
\text{EssFin}(\text{Op}_X)$ (cf. \cite{Pie1}). The \textbf{topologization}
of the gts $(X, \text{Cov}_X)$ is the topological space $X_{top}=(X,
\tau(\bigcup \text{Cov}_X))$.  \textbf{The topology induced by the generalized topology} 
$\text{Cov}_X$ is the collection  $\tau(\bigcup \text{Cov}_X)$.

Let $\tau_{nat}$ be the natural topology of the real line $\mathbb{R}$. We say that a subset $A$ of $\mathbb{R}$ is a \textbf{locally finite union of open intervals} if there exists a locally finite in $(\mb{R}, \tau_{nat})$ collection $\mc{U}$ of open intervals such that $A=\bigcup\mc{U}$. We will use the generalized topologies in $\mathbb{R}$ that induce $\tau_{nat}$ and are defined as follows: 

\begin{defi} \label{proste}
We give names to the following real lines:
\begin{enumerate}
\item[(i)]  the \textbf{usual topological real line} (in abbr. \textbf{the  ut-real line} or \textbf{the topological real line}) $\mb{R}_{ut}$
where $\text{Cov}=$ all families of members of $\tau_{nat}$;

\item[(ii)] the \textbf{o-minimal real line} $\mb{R}_{om}$
where $\text{Cov}=$ essentially finite families of finite unions of open intervals;

\item[(iii)] the \textbf{smallified topological} (or \textbf{small partially topological}) \textbf{real line} $\mb{R}_{st}$
where $\text{Cov}=$ essentially finite families of members of $\tau_{nat}$;

\item[(iv)] the \textbf{localized o-minimal real line} $\mb{R}_{lom}$
where $\text{Cov}=$ locally essentially finite families of locally finite unions of open intervals; 

\item[(v)] the \textbf{localized smallified topological real line} $\mb{R}_{lst}$
where $\text{Cov}=$ locally essentially finite families of members of $\tau_{nat}$;

\item[(vi)] the \textbf{smallified localized o-minimal real line} $\mb{R}_{slom}$ 
where $\text{Cov}=$ essentially finite families of locally finite unions of open intervals;

\item[(vii)] the \textbf{localized at $+\infty$ ($-\infty$, resp.) o-minimal real line} $\mb{R}_{l^+om}$ ($\mb{R}_{l^-om}$, resp.),
where $\text{Cov}=$ locally essentially finite families of locally finite unions of open intervals which, on the negative (positive, resp.) half-line,
are essentially finite and consist of only finite unions of open intervals; 

\item[(viii)] \textbf{localized at $+\infty$ ($-\infty$, resp.) smallified topological real line} $\mb{R}_{l^+st}$ ($\mb{R}_{l^-st}$, resp.)
where $\text{Cov}=$ locally essentially finite families of members of $\tau_{nat}$ which are essentially finite on the negative (positive, resp.) half-line; 

\item [(ix)] the \textbf{smallified localized at $+\infty$ ($-\infty$, resp.) o-minimal real line} $\mb{R}_{sl^+om}$ ($\mb{R}_{sl^-om}$, resp.) 
where $\text{Cov}=$ essentially finite families of locally finite unions of open intervals which are only finite unions of open intervals on the negative (positive, resp.) half-line;

\item[(x)] the \textbf{rationalized o-minimal real line} $\mb{R}_{rom}$ where $\text{Cov}=$ essentially finite families of finite unions of open intervals with endpoints being rational numbers or infinities.
\end{enumerate}
\end{defi}

Some of the generalized topologies in $\mb{R}$ from the definition above appeared in Example 2.2.14 of \cite{Pie1} and in Definition 2.1.15 of \cite{Pie2} when $n=1$; however, for elegance, we have changed the notation $\mb{R}_{ts}$ of \cite{Pie1} to $\mb{R}_{st}$. We have replaced the notation $\mb{R}_{top}$ of \cite{Pie1} by $\mb{R}_{ut}$ to reserve \textit{top} for other purposes. We shall also use modifications of the generalized topologies described in Definition 1.2.

\begin{defi}
Let $X$ be a topological space. Then:

\begin{enumerate}
\item[(i)] a \textbf{closed base} of $X$ is a collection $\mathcal{C}$ of closed sets in $X$ such that, for every closed set $A$ in $X$ and for
each $x\in X\setminus A$, there exists $C\in\mathcal{C}$ such that $x\in
X\setminus C$ and $A\subseteq C$;

\item[(ii)] a closed base $\mathcal{C}$ of $X$ will be called 
\textbf{complete} if $\emptyset, X\in\mathcal{C}$ and, simultaneously, if
the collection $\mathcal{C}$ is closed under finite unions and under finite
intersections.
\end{enumerate}
\end{defi}

Such collections of subsets of a set $X$ that are closed under finite unions
and under finite intersections are also called \textbf{rings of sets}. When $\mathcal{C}$ is a ring of subsets of a set $X$, then $\mathcal{C}\cup\{\emptyset\}$ is a closed base for a topology in $X$ and, therefore, we can call this ring \textbf{complete} if $\emptyset, X\in\mathcal{C}$.

\begin{defi}
Let $\mathcal{C}$ be a complete closed base of a topological space $X$. The
\textbf{gts induced by} $\mathcal{C}$ is the triple $(X, \text{Op}_X, \text{Cov}_X)$
where $\text{Op}_X=\{ X\setminus C: C\in\mathcal{C}\}$ and $\text{Cov}_X=
\text{EssFin}(\text{Op}_X)$.
\end{defi}

\begin{prop}
Let $X$ be a topological space. A collection $\mathcal{C}$ is a complete
closed base of $X$ if and only if there exists a generalized topology $\text{
Cov}_X$ in $X$ such that $\mathcal{C}$ is the collection of all closed sets
in the gts $(X,\text{Cov}_X)$.
\end{prop}

Let us slightly modify the standard notion of a Wallman base which is very
important in the theory of compactifications and in our present work. Sometimes,
Wallman bases are called normal bases (cf. e. g. \cite{Ch}, \cite{Fr} and \cite{PW}).

\begin{defi}
A \textbf{Wallman base} of a topological space $X$ is a closed base $\mathcal{C}$ of $X$ such that $\mathcal{C}$ satisfies the following conditions:

\begin{enumerate}
\item[(i)] $\mathcal{C}$ is closed under finite unions and under finite
intersections, i.e. $\mathcal{C}$ is a ring of sets,

\item[(ii)] for each set $A\subseteq X$ such that $A$ is a singleton or $A$
is closed in $X$, if $x\in X\setminus A$, then there exists $C\in\mathcal{C}$
such that $x\in C\subseteq X\setminus A$,

\item[(iii)] for every pair $A_1, A_2$ of disjoint members of $\mathcal{C}$,
there exists a pair $C_1, C_2$ of members of $\mathcal{C}$ such that $%
C_1\cup C_2=X$ and $A_i\cap C_i=\emptyset$ for each $i\in\{1, 2\}$.
\end{enumerate}
\end{defi}

Let us observe that every topological space which has a Wallman base is a $T_1$-space.
 Therefore, according to \cite{Fr}, we can call a topological space $X
$ \textbf{semi-normal} if $X$ has at least one Wallman base in the sense of
Definition 1.6. 

It is useful to make the convention that $P$ is a property of a
topology $\tau$ (or a generalized topology $\text{Cov}_X$, resp.) in X if and
only if the topological space $(X, \tau)$ (the gts $(X, \text{Cov}_X)$,
resp.) has property $P$. For example, a gts is
called \textbf{weakly normal} iff its generalized topology is weakly normal in the following sense adapted from Definition 2.2.81 of \cite{Pie1}:

\begin{defi}
We say that a generalized topology $\text{Cov}_X$ in $X$ is  \textbf{weakly normal} if, for
every pair $A_1, A_2$ of subsets of $X$ such that $A_i$ is a singleton or 
$A_i\in \text{Cl}_X$ for each $i\in\{1, 2\}$, there exists a pair $W_1, W_2$ of disjoint members of $\text{Op}_X$ such that $A_i\subseteq W_i$ for each $i\in\{1, 2\}$.
\end{defi}

\begin{prop}
Let $\mathcal{C}$ be a complete closed base of a topological space $X$. Then
the gts induced by $\mathcal{C}$ is weakly normal if and only if $\mathcal{C}
$ is a Wallman base of $X$.
\end{prop}
\begin{proof} We deduce it easily from Definitions 1.3, 1.6 and 1.7. 
\end{proof}

\begin{rem}
If $X$ is a non-empty semi-normal space, then every Wallman base of $X$ is a
complete closed base of $X$.
\end{rem}

In the light of Proposition 1.8, Wallman bases are so strictly
connected with weakly normal generalized topologies that it is interesting
to introduce a notion of a Wallman type compactification of a weakly normal
gts and create a reasonable theory of compactifications for generalized
topological spaces. This is the main topic of the paper although we hardly
start the theory and show directions of its future developments.

The basic theorem of the next section will help us to decide what axiom \textbf{A}
should be added to \textbf{ZF} to get our theorems about compactifications
provable in \textbf{ZF+A} if they are not provable in \textbf{ZF}. In the sequel, we will use the abbreviation \textbf{ZFU} for \textbf{ZF+UFT} and clearly denote the ones of our results  that are obtained in \textbf{ZFU} or in \textbf{ZFC} if it is not possible for us to deduce them in \textbf{ZF}. All other results of this work are in \textbf{ZF}. 

\section{Ultrafilter theorem via compactness in Wallman spaces}

To simplify our language and take care of its precision, we introduce the following notion:

\begin{defi} A \textbf{small space} is an ordered pair $(X, \mathcal{C})$ where $X$ is a set, while $\mathcal{C}$ is a complete ring of subsets of $X$.
\end{defi}

\begin{rem} By assigning to any small space $(X,\mathcal{C})$ the generalized topological space $(X, \text{Op}_X, \text{EssFin}(\text{Op}_X))$ such that $\mathcal{C}=\text{Cl}_X$, we obtain a one-to-one correspondence between small spaces in the sense of Definition 2.1 and these gtses that are small. Therefore, we identify a small space $(X, \mathcal{C})$ with the unique gts $(X, \text{Op}_X, \text{EssFin}(\text{Op}_X))$ such that $\mathcal{C}=\text{Cl}_X$. 
\end{rem}

The following modification of the notion of a disjunctive ring of \cite{Fr} will be useful.
 
\begin{defi} 
We say that a ring $\mathcal{C}$ of subsets of $X$ is \textbf{disjunctive} if, for each subset $A$ of $X$ and for each point $x\in X\setminus A$, when $A\in\mathcal{C}$ or $A$ is a singleton, then there exists $C\in \mathcal{C}$ such that $x\in C\subseteq X\setminus A$. A small space $(X,\mathcal{C})$ is called disjunctive when the ring $\mathcal{C}$ is disjunctive. 
\end{defi}

\begin{rem} Let $\mc{C}$ be a disjunctive ring of subsets of a set $X$. If $X$ consists of at least two points or if $\mathcal{C}$ is a closed base for a topology in a non-empty $X$, then $\emptyset\in\mathcal{C}$ but, in general, $X$ need not belong to $\mathcal{C}$. For example, the collection $\mathcal{C}$ of all finite subsets of an infinite set $X$ is a disjunctive ring such that $X\notin\mathcal{C}$.  
\end{rem}
In the following proposition, the notions of a weakly $T_1$ gts and of a weakly regular gts are taken from Definition 2.2.81 of \cite{Pie1}, while the notion of a $T_1$ closed base is given in \cite{HK}. 

\begin{prop} Let $(X, \mathcal{C})$ be a small space. If the gts induced by $\mathcal{C}$ is weakly regular, then $\mathcal{C}$ is disjunctive. If $\mathcal{C}$ is disjunctive, then the gts induced by $\mathcal{C}$ is weakly $T_1$. 
Moreover, $\mathcal{C}$ is a $T_1$ closed base of the topologization of the gts induced by $\mathcal{C}$ if and only if $\mathcal{C}$ is disjunctive. 
\end{prop}
\begin{proof} We leave it to readers as a simple deduction from appropriate definitions mentioned above.
\end{proof}

Let us pass to Wallman spaces.

\begin{defi}
Suppose that $(X, \mathcal{C})$ is a small space. Let $\mathcal{W}(X,\mathcal{C})$ be the collection of all ultrafilters in $\mathcal{C}$%
. For $A\in \mathcal{C}$, let $$[A]_{\mathcal{C}}=\{ p\in \mathcal{W}(X, \mathcal{C}):
A\in p\}.$$ 
\begin{enumerate}
\item[(i)] The set $\mathcal{W}(X, \mathcal{C})$ equipped with the topology which has
the collection $\{ [A]_{\mathcal{C}}: A\in\mathcal{C}\}$ as a closed base is
called the \textbf{Wallman space} corresponding to $\mathcal{C}$.
\item[(ii)] \textbf{The Wallman space of a gts} $Y$ is the Wallman space corresponding to $\text{Cl}_Y$.
\end{enumerate}
\end{defi}

It is possible that, even when $X$ is non-empty, the set $\mathcal{W}(X,\mathcal{C})$ is empty in a model for 
\textbf{ZF}. Namely, by Theorem 4.32 of \cite{Her} and its proof in \cite{Her},
 every model for $\mathbf{ZF+\neg AC}$ has a non-empty $T_0$-space $X$ such that there are no ultrafilters in the collection of all
closed sets of $X$. 

\begin{prop}
The following conditions are all equivalent:
\begin{enumerate}
\item[(i)] \textbf{AC} holds;
\item[(ii)] for every small space $(X,\mathcal{C})$ such that $X\neq\emptyset$, the space $\mathcal{W}(X,\mathcal{C})$ is non-void;
\item[(iii)] for every  disjunctive small space $(X,\mathcal{C})$, the space $\mathcal{W}(X,\mathcal{C})$ is compact.
\end{enumerate}
\end{prop}
\begin{proof} It suffices to apply Theorems 4.32 of \cite{Her} and 4.1 of \cite{KT} (cf. also Proposition 1 of \cite{HK}).
\end{proof}

Clearly, if $\mathcal{C}$ is a Wallman base of a non-empty $X$, then there are maximal fixed
filters in $\mathcal{C}$; hence, the Wallman space $\mathcal{W}(X, \mathcal{C})$
is non-empty in this case. In \cite{HK}, one can find several conditions that imply that $\mathcal{W}(\omega_1, \mathcal{C})$ is non-compact  where $\mathcal{C}$ is a collection of all closed sets of $\omega_1$ equipped with the order topology. In Solovay's model $\mathcal{M}5(\aleph )$ of \cite{HR}, the space $\mathcal{W}(\omega, \mathcal{P}(\omega))$ is non-compact. The theorem below shows that, in every model for $\mathbf{ZF+\neg UFT}$, there exists a small space $(X, \mathcal{C})$ such that $\mathcal{W}(X, \mathcal{C})$ is non-compact. 

\begin{thm}
\textbf{UFT} holds if and only if, for
every semi-normal space $X$ and for each Wallman base $\mathcal{C}$ of $X$,
the Wallman space $\mathcal{W}(X, \mathcal{C})$ is compact.
\end{thm}

\begin{proof} \textsl{Necessity.} First, assume that \textbf{UFT} holds and suppose that $\mathcal{C}$ is a Wallman base of a non-void semi-normal space $X$. Consider a centred family $\mathcal{H}\subseteq\{ [A]_{\mathcal{C}}: A\in\mathcal{C}\}$. By the theorem given in Exercise 5 of Section 4.3 in \cite{Her}, there exists a prime filter $\mathcal{F}$ in $\mathcal{C}$ such that $\{A\in\mathcal{C}: [A]_{\mathcal{C}}\in\mathcal{H}\}\subseteq\mathcal{F}$. Let $$\mathcal{M}=\{ A\in\mathcal{C}:\forall_ {F\in\mathcal{F}}A\cap F\neq\emptyset\}.$$  To prove that $\mathcal{M}$ is a filter in $\mathcal{C}$, consider an arbitrary pair $A_1, A_2$ of members of $\mathcal{M}$. Suppose that $A_1\cap A_2\notin\mathcal{M}$. There exists $F\in\mathcal{F}$ such that $A_1\cap A_2\cap F=\emptyset$. Since $\mathcal{C}$ is a Wallman base, there exists a pair $C_1, C_2$ of members of $\mathcal{C}$ such that $C_1\cup C_2=X$, $A_1\subseteq X\setminus C_1$ and $A_2\cap F\subseteq X\setminus C_2$. That $\mathcal{F}$ is prime implies that $C_1\in\mathcal{F}$ or $C_2\in\mathcal{F}$. However, $C_1\notin\mathcal{F}$ because $A_1\in\mathcal{M}$ and $A_1\cap C_1=\emptyset$. If $C_2\in\mathcal{F}$, then $F\cap C_2\in\mathcal{F}$ but this is impossible because $A_2\in\mathcal{M}$ does not meet $F\cap C_2$. It follows from the contradiction obtained that $A_1\cap A_2\in\mathcal{M}$. This implies that $\mathcal{M}$ is a filter in $\mathcal{C}$. Of course, the filter $\mathcal{M}$ is maximal in $\mathcal{C}$. In consequence, $\mathcal{M}\in \mathcal{W}(X,\mathcal{C})$ and $\mathcal{M}\in\bigcap\mathcal{H}$. 

\textsl{Sufficiency.} Now, assume that \textbf{UFT} does not hold. In view of Theorem 4.37 of \cite{Her}, there exists a product $X=\prod_{s\in S}X_s$ of finite discrete spaces $X_s$ such that $X$ is not compact. Let $\mathcal{C}=\mathcal{Z}(X)$ be the collection of all zero-sets in $X$. Since $X$ is not compact, there exists a filter $\mathcal{F}$ in $\mathcal{C}$ such that $\bigcap\mathcal{F}=\emptyset$. Suppose that the Wallman space $\mathcal{W}(X, \mathcal{C})$ is compact. There is $p\in \mathcal{W}(X,\mathcal{C})$ such that $p\in\bigcap\{[A]_{\mathcal{C}}: A\in\mathcal{F}\}$. Then $p$ is an ultrafilter in $\mathcal{C}$ such that $\mathcal{F}\subseteq p$. Let $s\in S$ and let $\pi_s: X\to X_s$ be the projection. Then $\mathcal{U}_s=\{A\subseteq X_s: {\pi_s}^{-1}(A)\in p\}$ is an ultrafilter in $\mathcal{P}(X_s)$; therefore, there exists a unique $x_s\in\cap\mathcal{U}_s$. Then $x\in X$ such that $\pi_s(x)=x_s$ for each $s\in S$ is a point from $\bigcap\mathcal{F}$. The contradiction obtained shows that the space $\mathcal{W}(X,\mathcal{C})$ is not compact. 
\end{proof}

So far as compactifications are concerned, we are interested mainly in Hausdorff compactifications, among them, in compactifications of Wallman type that correspond to Wallman bases. Recall that, in \textbf{ZFC}, not all Hausdorff compactifications of Tychonoff spaces are of Wallman type (cf. \cite{U}).

There is a model for $\mathbf{ZF +\neg AC}$ in which there exists a semi-normal space which is not completely regular (cf. e.g. \cite{Jech1}). There are models for $\mathbf{ZFU+\neg AC}$ (cf. e.g. \cite{Her} and \cite{HR}). However, we are unable to give a satisfactory answer to the following question:

\begin{q} 
Must every semi-normal space be completely regular in \textbf{ZFU}?
\end{q}
  
\section{Compactness for generalized topologies}

To begin a theory of compactifications of gtses, we must define compactness in the category $\mathbf{GTS}$ of generalized topological spaces first. 

\begin{defi}
Let $P$ be a topological property. We will say that a gts $(X,
\text{Op}_X,\text{Cov}_X)$ \textbf{has $P$ topologically} if its topologization has $P$. We will say that $(X,\text{Op}_X,\text{Cov}_X)$ \textbf{%
has $P$ admissibly} if any instance of "open family" in the definition of $P$
is replaced by "admissible open family". We will say that $(X,\text{Op}_X,\text{Cov}_X)$
\textbf{has $P$ absolutely} if the definition of $P$ is imported verbatim in 
\textbf{GTS}. 
\end{defi}

\begin{defi}
Let $X$ be a gts. A set $Y\subseteq X$ is called \textbf{topologically} (\textbf{absolutely}, \textbf{admissibly}, resp.) \textbf{compact} in $X$ if, for every weakly open (open, admissibly open, resp.) in $X$ family $\mathcal{V}$ such that $Y\subseteq\bigcup\mathcal{V}$, there exists a finite family $\mathcal{U}\subseteq\mathcal{V}$ such that $Y\subseteq\bigcup\mathcal{U}$.
\end{defi}

\begin{defi}
A gts $X$ is called \textbf{topologically} (\textbf{absolutely}, \textbf{admissibly}, resp.) \textbf{compact} if $X$ is topologically  (absolutely, admissibly, resp.) compact in $X$.
\end{defi}

\begin{prop}
Let $Y$ be a subset of a gts $X$. Then $Y$ is topologically compact in $X$ if and only if $Y$ is absolutely compact in $X$. Moreover, if $Y$ is absolutely compact in $X$, then $Y$ is admissibly compact in $X$.
\end{prop}

\begin{proof}
Since each admissible open family is an open family, absolute compactness implies admissible compactness. We include a proof for less advanced readers that absolute compactness implies topological compactness. If $Y$ is absolutely compact in $X$ and a weakly open in $X$ family $\mathcal{V}$ covers $Y$, to avoid \textbf{AC}, for each $V\in\mathcal{V}$,  we define $O(V)=\{ U\in\text{Op}_X: U\subseteq V\}$ and consider the open family $\mathcal{W}=\bigcup \{ O(V): V\in\mathcal{V}\}$. If $\mathcal{U}\subseteq \mathcal{W}$ is finite and covers $Y$, to each $U\in\mathcal{U}$, we assign one $V(U)\in\mathcal{V}$ such that $U\subseteq V(U)$. Then $\{ V(U): U\in \mathcal{U}\}$ is finite and covers $Y$. This shows that absolute compactness in $X$ implies topological compactness in $X$. Without the convention that, for every non-void finite family of pairwise disjoint non-void sets, there is in \textbf{ZF} a choice function of this family, we are unable to prove that topological compactness follows from absolute compactness. Of course, absolute compactness follows from topological compactness.  
\end{proof}

\begin{exam}
Every subspace a small gts is admissibly compact. In particular, all small real lines described in Definition 1.2 (ii), (iii), (vi), (ix) and (x) are admissibly compact but not absolutely compact.
\end{exam}

\begin{defi} Let $\text{Cov}_X$ be a generalized topology in $X$ and let $Y\subseteq X$. We say that $Y$ is a \textbf{strict subspace} of the gts $(X, \text{Cov}_X)$ or, equivalently, that the set $Y$ is strict in $X$ if ${\langle \text{Cov}_X{\cap}_2 Y\rangle}_Y= \text{Cov}_X {\cap}_2 Y$ (cf. \cite{Pie1}). 
\end{defi}

\begin{rem} If $Y$ is a strict subset of a gts $(X,\text{Op}_X, \text{Cov}_X)$, then $\text{Op}_Y=\bigcup (\text{Cov}_X{\cap}_2 Y)=\text{Op}_X{\cap}_1 Y$. We still do not know whether every subset of a gts $X$ is strict in $X$ (cf. Question 2.2.96 of \cite{Pie1}).    
\end{rem}

A gts $X$ is called \textbf{topological} if $\text{Cov}_X=\mathcal{P}(\text{Op}_X)$ (cf. Definition 2.2.67 of \cite{Pie1}).

\begin{f}
Every subspace of a topological gts $X$ is a topological gts and it is strict in $X$. 
\end{f}

\begin{f} Let $Y$ be a subspace of a gts $X$. Then the topology in $Y$ generated by $\text{Op}_Y$ is identical with the topology in $Y$ generated by $\text{Op}_X{\cap}_1 Y$ and with $\tau(Op_X)\cap_1 Y$. In this sense, the topologization of $Y$ is a topological subspace of the topologization of $X$.
\end{f} 

With Proposition 3.4 in hand, we can deduce the following:

\begin{cor}
Let $Y$ be a subspace of a gts $X$. Then the following properties hold:
\begin{enumerate}
\item[(i)] $Y$ is absolutely compact in $Y$ if and only if $Y$ is absolutely compact in $X$; 
\item[(ii)] when $Y$ is strict in $X$, then $Y$ is admissibly compact in $Y$ if and only if $Y$ is admissibly compact in $X$. 
\end{enumerate}
\end{cor}

\begin{rem} 
In the light of Proposition 3.4 and Corollary 3.10, the notions of topological and absolute compactness are equivalent in the class of generalized topological spaces, therefore, we can call a subspace of a gts absolutely compact iff it is topologically compact in itself. Moreover, we can call a strict subspace of a gts admissibly compact iff this subspace is admissibly compact in itself. Since, in category theory, there is a notion of compact categories (cf. \cite{AHS}), to avoid misunderstanding, we would not like to use the term "compact" in the class of gtses as a synonym of "topologically compact". 
\end{rem}  

\begin{f}
A closed subspace of an admissibly (topologically, resp.)  compact gts is admissibly (topologically, resp.) compact.
\end{f}

\begin{f}
A finite union of closed admissibly (topologically, resp.) compact subspaces of a gts $X$ is a
closed admissibly (topologically, resp.) compact subspace of $X$.
\end{f}

\begin{f}
If $\mathcal{F}$ is an admissible closed family (see Remark 2.2.55 in \cite{Pie1}) of a gts $X$ and there is a topologically compact in $X$ member of $\mathcal{F}$, then $\bigcap\mathcal{F}$ is closed and topologically compact in $X$.
\end{f}

Let us notice that a gts is  \textbf{weakly Hausdorff} in the sense of \cite{Pie1} if and only if it is topologically Hausdorff.

\begin{prop}
Each weakly Hausdorff topologically compact gts is weakly normal.
\end{prop}

\begin{proof}
The standard topological proof, if carefully led in \textbf{ZF}, extends to this situation.
\end{proof}

\section{Products of generalized topological spaces}

Morphisms in the category $\mathbf{GTS}$ are strictly continuous mappings where a mapping $f:X\to Y$ from a gts $X$ to a gts $Y$ is called \textbf{strictly continuous} if $f^{-1}(\mathcal{V})\in\text{Cov}_X$  for each $\mathcal{V}\in\text{Cov}_Y$ (cf. e.g. Definition 2.2.4 of \cite{Pie1}). 
A gts $X$ is \textbf{partially topological} when $\text{Op}_X$ is the topology of $X_{top}$ (cf. Definition 2.2.67 of \cite{Pie1}). The partially topological gtses form a full subcategory $\mathbf{GTS}_{pt}$ of $\mathbf{GTS}$ (cf. \cite{Pie1}). 

\begin{defi}
For a gts $X=(X, \text{Op}_X, \text{Cov}_X)$, let $(\text{Op}_{X})_{pt}$ be the topology  $\tau(\text{Op}_X)$ in $X$ generated by $\text{Op}_X$ and let $(\text{Cov}_{X})_{pt}$ be the generalized topology $\langle \text{Cov}_X\cup\text{EssFin}(\tau(\text{Op}_X))\rangle_X$. Then the gts $X_{pt}=(X,(\text{Op}_X)_{pt}, (\text{Cov}_X)_{pt})$ is called the \textbf{partial topologization} of $X$. 
\end{defi}

\begin{defi} 
The \textbf{functor of partial topologization} is the mapping $pt:\mathbf{GTS}\to\mathbf{GTS}_{pt}$ such that, when $f$ is a morphism in $\mathbf{GTS}$ and $X$ is a gts, then $pt(f)=f$ and $pt(X)=X_{pt}$.
\end{defi}
   
\begin{defi} 
Suppose we are given a class $I$, a set $X$ and, for every $i\in I$, a gts $Y_i$ with its generalized topology $\text{Cov}_i$. Let $F=\{ f_i: i\in I\}$ where $f_i: X\to Y_i$ for every $i\in I$. We call the collection $\text{Cov}_{X}(F)=\langle \bigcup_{i\in I} {f_i}^{-1}(\text{Cov}_i)\rangle_X$ the \textbf{$\mathbf{GTS}$-initial generalized  topology} in $X$ for the class $F$. Then the generalized topology  $(\text{Cov}_{X}(F))_{pt}$ is called the \textbf{$\mathbf{GTS}_{pt}$-initial generalized topology} in $X$ for $F$. 
\end{defi}

The proofs in \cite{Pie1} that the categories $\mathbf{GTS}$ and $\mathbf{SS}$ are topological are correct in \text{ZF+[AC for classes]} but, unfortunately, they are incorrect in \textbf{ZF}. One can easily improve them in \textbf{ZF} by mimicking part of our proof to the following theorem.  

\begin{thm}
The construct $\mathbf{GTS}_{pt}$ is topological.
\end{thm}

\begin{proof} Let $F=\{f_i: i\in I\}$ be a source of mappings $f_i:X\to Y_i$ indexed by a class $I$ (cf. \cite{AHS}). Assume that each $Y_i$ has a generalized topology $\text{Cov}_i$. Let us give $X$ the $\mathbf{GTS}$-initial topology for $F$.  Clearly, $X_{pt}$ is an object of $\mathbf{GTS}_{pt}$, and $id:X_{pt}\to X$ is the canonical morphism, with all $f_i\circ id$  morphisms in $\mathbf{GTS}_{pt}$.  The class $I$ can be proper and, without \textbf{AC} for classes, we should not assume that $I$ is a set nor that $I$ can be replaced by a set; however, in much the same way,  as in the proof of Theorem 2.2.60 in \cite{Pie1}, we can observe that, for any object $Z$ of $\mathbf{GTS}_{pt}$ and a mapping $h:Z\to X_{pt}$, if all $f_i\circ id \circ h$ with $i\in I$ are morphisms, then 
 $id\circ h:Z \to X$ is a morphism in $\gts$, so $h=(id \circ h)_{pt}$ is a morphism in $\gts_{pt}$. 
\end{proof}

\begin{rem} Since the category $\mathbf{GTS}$ is fibre-small, one can try to prove Theorem 4.4 by applying Proposition 21.36 of \cite{AHS}; however, this is incorrect in \textbf{ZF} because the arguments given in \cite{AHS} that, if a structured source in a fibre-small concrete category is indexed by a large class $I$, then $I$  can be replaced by a set $J\subseteq I$ require both \textbf{AC} for classes and replacement scheme. Therefore, an interesting question is whether Proposition 21.36 of \cite{AHS} can be false in a model for \textbf{ZF}. 
\end{rem}  

Now, we are in a position to define products of gtses. 

\begin{defi} Let $S$ be a set and let $\{X_s: s\in S\}$ be a collection of sets $X_s$. Assume that $\text{Cov}_s$ is a generalized topology in $X_s$ where $s\in S$. Let $X=\prod_{s\in S} X_{s}$ and, for each $s\in S$, let ${\pi}_s: X\to X_s$ be the projection. In what follows,  we treat each $X_s$ as the gts with its generalized topology $\text{Cov}_s$. Then:
\begin{enumerate}
\item[(i)] the \textbf{$\mathbf{GTS}$-product generalized topology} in the set $X$ is the $\mathbf{GTS}$-initial generalized topology in $X$ for the collection $\{ {\pi}_s: s\in S\}$;
\item[(ii)]  the \textbf{$\mathbf{GTS}$-product of the collection} $(X_s, \text{Cov}_s)$ with $s\in S$ is the gts $(X, \langle \bigcup_{s\in S}{{\pi}_s}^{-1}(\text{Cov}_s)\rangle_X)$;
\item[(iii)] when all gtses $(X_s, \text{Cov}_s)$ with $s\in S$ are partially topological, the \textbf{$\mathbf{GTS}_{pt}$-product generalized topology} in the set $X$ is the $\mathbf{GTS}_{pt}$-initial generalized topology in $X$ for the collection $\{ {\pi}_s: s\in S\}$;
\item[(iv)] when all gtses $(X_s, \text{Cov}_s)$ with $s\in S$ are partially topological,  the \textbf{$\mathbf{GTS}_{pt}$-product} of their collection is the gts $$(X,( \langle \bigcup_{s\in S}{{\pi}_s}^{-1}(\text{Cov}_s)\rangle_X)_{pt}).$$ 
\end{enumerate}
\end{defi}

The following example shows that the $\mathbf{GTS}$-product and the $\mathbf{GTS}_{pt}$-product of the same collection of partially topological gtses can be different.

\begin{exam}
For the smallified topological real line $\mathbb{R}_{st}$ (cf. Definition 1.2 (iii)), let $X_1=X_2=\mathbb{R}_{st}$. Then the $\mathbf{GTS}$-product of the collection $\{X_1, X_2\}$ is a small gts which is not partially topological. Indeed, if $H=\{(x_1, x_2)\in X_1\times X_2: x_1\cdot x_2=1\}$, then the set $U=(X_1\times X_2)\setminus H$ is not open in the $\mathbf{GTS}$-product of $\{ X_1, X_2\}$, while $U$ is open in $\mathbf{GTS}_{pt}$-product of $\{X_1, X_2\}$. 
\end{exam}  
We encourage readers to see what products as sources in category theory are in \cite{AHS}.
\begin{rem}Suppose that $\mathbf{C}$ is a full subcategory of the category $\mathbf{GTS}$ such that $\mathbf{C}$ has products for all set-indexed families. 
\begin{enumerate} 
\item[(i)] Sometimes, we will use ${\prod^{\mathbf{C}}_{s\in S}}X_s$ to denote the $\mathbf{C}$-product of the collection $\{X_s: s\in S\}$ of objects of the category $\mathbf{C}$. Given a pair $X, Y$ of gtses in $\mathbf{C}$, we may write $X\times_{\mathbf{C}}Y$ to denote the $\mathbf{C}$-product of the collection $\{X , Y\}$.  
\item[(ii)] Since $\mathbf{GTS}$-products of small gtses are also small gtses, the notions of $\mathbf{SS}$-products and of $\mathbf{GTS}$-products of small gtses are identical. 
\item[(iii)] The category of all partially topological small gtses is denoted by $\mathbf{SS}_{pt}$. If $\mathcal{X}=\{(X_s, \text{Cov}_s): s\in S\}$ is a set-indexed collection of partially topological small gtses, $X=\prod_{s\in S}X_s$ and $\text{Op}_X$ is the collection of all open sets of the $\mathbf{GTS}_{pt}$ -product of $\mathcal{X}$, then $\text{EssFin}( \text{Op}_X)$ is the generalized topology of the $\mathbf{SS}_{pt}$-product of $\mathcal{X}$.
\item[(iv)] It is unclear whether every $\mathbf{GTS}$-product of topological gtses is a topological gts; however, if this is necessary, one can use the functor $\text{top}:\mathbf{GTS}\to\mathbf{Top}$ (cf. Definition 2.2.64 of \cite{Pie1}) to transform $\mathbf{GTS}$-products of topological gtses into $\mathbf{Top}$-products. 
\end{enumerate}
\end{rem}

\begin{prop} Let $X$ and $Y$ be topological discrete gtses. Then the $\mathbf{GTS}$-product $X\times_{\mathbf{GTS}}Y$ of $X$ and $Y$ is a topological discrete gts.
\end{prop}
\begin{proof}
Let us observe that the collection of all singletons of $X\times Y$ is an admissible open family of $X\times_{\mathbf{GTS}}Y$. Thus, by  (A8) of Definition 2.2.1 of \cite{Pie1}, each subset $W\subseteq X\times Y$ is open in $X\times_{\mathbf{GTS}}Y$.

Let $\mc{U}$ be a family of subsets of $X\times Y$. In view of condition (A5) of Definition 2.2.1 of \cite{Pie1}, since $\bigcup\mc{U}$ is open in $X\times_{\mathbf{GTS}}Y$, the covering of $\bigcup\mc{U}$ by singletons is admissible in $X\times_{\mathbf{GTS}}Y$. The family $\mc{U}$ is a coarsening of the previous one. By (A7) of Definition 2.2.1 of \cite{Pie1}, $\mc{U}$ is admissible in $X\times_{\mathbf{GTS}}Y$.
\end{proof}

\begin{f} 
Let $\mathbf{C}$ be one of the categories $\mathbf{GTS}, \mathbf{SS}, \mathbf{GTS}_{pt}, \mathbf{SS}_{pt}$. 
Suppose that $\mathcal{X}=\{ X_s: s\in S\}$ is a family of gtses $X_s$ that are objects of $\mathbf{C}$ for $s\in S$. Then $(\prod^{\mathbf{C}}_{s\in S}X_s)_{top}$ is the Tychonoff product $\prod_{s\in S}(X_s)_{top}$, i.e. the topologization of a $\mathbf{C}$-product of gtses is the Tychonoff product of the topologizations of the gtses.
\end{f}

The beautiful equivalence of \textbf{UFT} with the statement that all Tychonoff products of compact Hausdorff spaces are compact (cf. Theorem 4.37 of \cite{Her}) can be adapted to the language of $\mathbf{GTS}$ as follows:

\begin{thm} 
That all $\mathbf{GTS}$-products of weakly Hausdorff topologically compact gtses are topologically compact is equivalent with \textbf{UFT}.
\end{thm} 
 
Since we identify $X\times Y$ with the product $\prod_{Z\in\{ X, Y\}} Z$, we use the symbols ${\pi}_X$ and ${\pi}_Y$ to denote the projections ${\pi}_X:X\times Y\to X$ and ${\pi}_Y:X\times Y\to Y$, respectively.

\begin{rem}
Looking at the proof of Kuratowski's Theorem 3.1.16 in \cite{En}, we can deduce that it is valid in \textbf{ZF} that a topological space $X$ is compact if and only if, for every topological space $Y$, the projection ${\pi}_Y: X\times Y\to Y$ is a closed mapping. Let us observe that \textbf{AC} is strongly involved in Mr\'owka's proof to this theorem given in \cite{Mr}.
\end{rem}
In the language of $\mathbf{GTS}$, we can get versions of Kuratowski-Mr\'owka theorem for gtses in \textbf{ZF}.

\begin{thm}
For every partially topological gts $X$, the following conditions are
equivalent:
\begin{enumerate}
\item[(i)] $X$ is topologically compact,
\item[(ii)]$X$ is $\mathbf{GTS}_{pt}$--complete (cf. Definition 4.1.1. of \cite{Pie2}).
\end{enumerate}
\end{thm}

\begin{proof}
Let us consider an arbitrary pair $X,Y$ of objects of  $\gts_{pt}$. The open sets in their product $X\times_{\mathbf{GTS}_{pt}} Y$ in 
$\mathbf{GTS}_{pt}$ are all unions of open boxes.
If $X$ is topologically compact, then for a closed $C$ in $X\times_{\mathbf{GTS}_{pt}} Y$ its projection ${\pi}_Y (C)$ is weakly closed, so closed.
If $X$ is  $\mathbf{GTS}_{pt}$--complete, then $X_{top}$ is $\mathbf{Top}$--complete (cf. Definition 4.1.1 of \cite{Pie2}), and $X$ is topologically compact by 
Remark 4.12.
\end{proof}

\begin{defi}
A gts $X$ will be called \textbf{topologically}  $\mathbf{GTS}$--\textbf{complete} if, for any object $Y$ in $\mathbf{GTS}$,  the projection ${\pi}_Y (C)$ of a weakly closed set $C$ in the product $X\times_{\mathbf{GTS}} Y$ is a weakly closed set in $Y$.
\end{defi}

The next theorem  is inspired by Question 4.1.3 of \cite{Pie2}.

\begin{thm}
A gts $X$ is topologically compact if and only if it is topologically 
$\mathbf{GTS}$--complete.
\end{thm}

\begin{proof}
It suffices to observe that the condition that $X$ is topologically compact means exactly that $X_{top}$ is $\mathbf{Top}$--complete and this is equivalent to the condition that each weakly closed subset of the product $X\times_{\mathbf{GTS}} Y$ has a weakly closed projection in $Y$.
\end{proof}
 
\section{Strict compactifications in $\mathbf{GTS}$}

To start our theory of strict compactifications of gtses, we need the following notions:
 
\begin{defi} For a pair $X, Y$ of gtses, a mapping $f:X\to Y$ is called:
\begin{enumerate}
\item[(i)] an \textbf{embedding} of $X$ into $Y$ if $f$ is a strictly continuous injection such that $f(\text{Cov}_X)\subseteq \langle \text{Cov}_Y{\cap}_2 f(X)\rangle_{f(X)}$;
\item[(ii)] a \textbf{strict embedding} of $X$ into $Y$  if $f$ is an embedding of $X$ into $Y$ such that $ f(\text{Cov}_X)\subseteq \text{Cov}_Y{\cap}_2 f(X)$.
\end{enumerate}
\end{defi}

\begin{defi} A \textbf{strict compactification} of a gts $X$ is an ordered pair $(\alpha X, \alpha )$ such that $\alpha X$ is a topologically compact gts and $\alpha: X\to \alpha X$ is a strict embedding such that $\alpha (X)$ is a dense subset of the topologization $(\alpha X)_{top}$ of $\alpha X$.
\end{defi}

\begin{rem} Let $\alpha X, \gamma X$ be Hausdorff compactifications of a topological space $X$. If there exists a homeomorphism $h:\alpha X\to\gamma X$ such that $h\circ \alpha =\gamma$, let us say that $\alpha X$ and $\gamma X$ are topologically equivalent and write $\alpha X=\gamma X$. When there is a continuous function $f:\alpha X\to \gamma X$ such that $f\circ\alpha =\gamma$, let us use the notation $\gamma X\leq\alpha X$ and say that $\gamma X$ is topologically smaller than $\alpha X$. Compactifications of topological spaces will be called \textbf{topological compactifications}.
\end{rem}

\begin{rem}
Let $X$ be a gts. In much the same way, as in the classical theory of compactifications of topological spaces, we denote a strict compactification $(\alpha X, \alpha)$ by $\alpha X$ from time to time. When $\alpha X$ and $\gamma X$ are strict compactifications of $X$, one can write $\alpha X\preceq \gamma X$ if and only if there exists a strictly continuous surjection $f:\gamma X\to \alpha X$ such that $f\circ \gamma = \alpha$. Let us call strict compactifications $\alpha X$ and $\gamma X$ of $X$ \textbf{strictly equivalent} if there exists a strict homeomorphism $h:\alpha X\to \gamma X$ such that $h\circ \alpha = \gamma$. It is clear that if $\alpha X$ and $\gamma X$ are weakly Hausdorff strict compactifications of $X$, then  $\alpha X$ and $\gamma X$ are strictly equivalent if and only if $\alpha X\preceq \gamma X$ and $\gamma X\preceq \alpha X$. We write $\alpha X\approx \gamma X$ when $\alpha X$ and $\gamma X$are strictly equivalent strict compactifications.
When $\alpha X$ is a strict compactification of a gts $X$ or a topological compactification of a topological space $X$, we identify each point $x\in X$ with $\alpha(x)$, the mapping $\alpha$ with $\text{id}_X$, $\alpha(X)$ with $X$ and, of course, the remainder $\alpha X\setminus \alpha(X)$ with $\alpha X\setminus X$.
\end{rem}
 
\begin{defi} 
Let $\alpha X$ be a topological compactification of the topologization $X_{top}$ of a gts $X$ and let $\tau_{\alpha X}$ be the topology of $\alpha X$. Define 
$$\text{Op}^{S}_{\alpha X}= \{ V\in {\tau}_{\alpha X}: {\alpha}^{-1}(V)\in\text{Op}_X\}$$ and $$\text{Cov}^{S}_{\alpha X}=\{ \mathcal{V}\subseteq \text{Op}^{S}_{\alpha X}: {\alpha}^{-1}(\mathcal{V})\in \text{Cov}_X\}.$$ The collection $\text{Cov}^{S}_{\alpha X}$ will be called \textbf{the strongest generalized topology in $\alpha X$ associated with} $\text{Cov}_X$. The gts $${\alpha}^{S}X=((\alpha X, \text{Op}^{S}_{\alpha X}, \text{Cov}^{S}_{\alpha X}), \alpha)$$ will be called \textbf{the strongest strict compactification of $X$ corresponding to $\alpha X$}.
\end{defi}

\begin{prop} Let $X$ be a gts and let $\alpha X$ be a compactification of the topologization $X_{top}$ of $X$. Then the following hold true:
\begin{enumerate}
\item[(i)] ${\alpha}^{S}X$ is a strict compactification of the gts $X$; 
\item[(ii)] if $\tau (\text{Op}^{S}_{\alpha X})$ is Hausdorff, then ${\tau}_{\alpha X}=\tau (\text{Op}^{S}_{\alpha X})$;
\item[(iii)] if $X\in\tau_{\alpha X}$, then the remainder $Y=\alpha X\setminus X$ as a subspace of the gts ${\alpha}^{S}X$ is a topological gts.
\end{enumerate}
\end{prop}
\begin{proof} It is easy to check that (i) holds. To prove (ii), it is sufficient to apply the fact that, among Hausdorff topologies in $\alpha X$, compact topologies are minimal (cf. Corollary 3.1.14 of \cite{En}). Now, assume that $X\in{\tau}_{\alpha X}$. For each $V\in{\tau}_{\alpha X}$, the set $V\cup X$ is in $\text{Op}^{S}_{\alpha X}$, so $V\setminus X= V\cap Y\in\text{Op}_Y$. This implies that ${\tau}_{\alpha X}\cap_1 Y=\text{Op}_Y$. If $\emptyset \neq \mc{U}\subseteq \text{Op}_Y$, then $\mc{U}\cup_{1}X\subseteq\text{Op}^{S}_{\alpha X}$ and $(\mc{U}\cup_1 X)\cap_1 X=\{ X\}\in\text{Cov}_X$, so $\mc{U}\cup_1 X\in\text{Cov}^{S}_{\alpha X}$. Since $(\mc{U}\cup_1 X)\cap_1 Y=\mc{U}$, we have that $\mc{U}\in\text{Cov}_Y$. This is why $\text{Cov}_Y=\mc{P}(\text{Op}_Y)$. In consequence, (iii) holds.
\end{proof}
Let us give an example showing that the topology $\tau (\text{Op}^{S}_{\alpha X})$ on $\alpha X$ generated by $\text{Op}^{S}_{\alpha X}$ need not be equal to ${\tau}_{\alpha X}$.
\begin{exam} 
For $X=\omega$, the set $\alpha X=\omega + 1$ equipped with the topology $\tau = \mathcal{P}(X)\cup\{ U\subseteq \alpha X: \omega\in U \text{and} \mid X\setminus U\mid <\omega\}$ is Alexandoff's compactification of the discrete subspace $X$ of $\alpha X$. Let $\text{Op}_X=\{X\}\cup\{A\subseteq X: \mid A\mid <\omega\}$ and $\text{Cov}_X=\text{EssFin}(\text{Op}_X)$. Then $\tau$ is different from the topology in $\alpha X$ generated by $\text{Op}^{S}_{\alpha X}$ where $\text{Op}^{S}_{\alpha X}= \{ V\in\tau: V\cap X\in\text{Op}_X\}$.
\end{exam}

\begin{defi} 
Let $X$ be a gts and let $\alpha X$ be a topological compactification  of  $X_{top}$.  We define an operator $\text{Ex}_{\alpha X}:\text{Op}_X\to \text{Op}^{S}_{\alpha X}$ by $$\text{Ex}_{\alpha X}(U)=\alpha X\setminus\text{cl}_{\alpha X}( X\setminus U)$$ for each $U\in\text{Op}_X$.
 The collection $\text{Op}^{w}_{\alpha X}=\{ \text{Ex}_{\alpha X}(U): U\in\text{Op}_X\}$ will be called \textbf{the wallmanian part} of $\text{Op}^{S}_{\alpha X}$, while the collection $$\text{Cov}^{w}_{\alpha X}=\{ \mathcal{V}\in\text{Cov}^{S}_{\alpha X}: \mathcal{V}\subseteq \text{Op}^{w}_{\alpha X}\}$$ will be called the wallmanian part of $\text{Cov}^{S}_{\alpha X}$. 
\end{defi}

\begin{rem}
In the definition above,  $\text{Ex}_{\alpha X}(U)$ is the biggest (with respect to inclusion) set $V\in\text{Op}^{S}_{\alpha X}$ such that $V\cap X=U$ (cf. Section 7.1 of \cite{En}). Furthermore, we have $\bigcup\text{Cov}^{w}_{\alpha X}=\text{Op}^{w}_{\alpha X}$ and $\text{Op}^{w}_{\alpha X}$ is stable under finite intersections;  however, it may happen that $\text{Op}^{w}_{\alpha X}$ is not stable under finite unions.
\end{rem}

\begin{exam}
Let $X$ be the small partially topological real line $\mathbb{R}_{st}$ (cf. Definition 1.2 (iii)). Let $\alpha X$ be the Alexandroff topological compactification of $X_{top}$. Then $\text{Ex}_{\alpha X}((-\infty; 0))\cup\text{Ex}_{\alpha X}((0; +\infty))\notin\text{Op}^{w}_{\alpha X}$.
\end{exam}

\begin{defi} Let $X$ be a gts. For a topological compactification $\alpha X$ of $X_{top}$, let us say that the operator $\text{Ex}_{\alpha X}$ is:
\begin{enumerate}
\item[(i)] \textbf{finitely additive} if $\text{Op}^w_{\alpha X}$ is stable under finite unions;
\item[(ii)] \textbf{admissibly additive} if, for each $\mathcal{U}\in\text{Cov}_X$, the following equality holds: $$\text{Ex}_{\alpha X}(\bigcup_{U\in\mathcal{U}}U)=\bigcup_{U\in\mathcal{U}}\text{Ex}_{\alpha X}(U).$$
\end{enumerate}
\end{defi}

\begin{prop} 
When $X$ is a gts and $\alpha X$ is a compactification of $X_{top}$, then $\text{Cov}^{w}_{\alpha X}$ is a generalized topology in $\alpha X$ if and only if the operator $\text{Ex}_{\alpha X}$ is admissibly additive.
\end{prop}
\begin{proof}
If $\text{Cov}^{w}_{\alpha X}$ is a generalized topology in $\alpha X$, then it follows from condition (A4) of Definition 2.2.1 of \cite{Pie1} that $\text{Ex}_{\alpha X}$ is admissibly additive. On the other hand, if $\text{Ex}_{\alpha X}$ is admissibly additive, it is not hard to check that $\text{Cov}^w_{\alpha X}$ satisfies conditions (A1)-(A8) of Definition 2.2.1 of \cite{Pie1} or the conditions of Definition 2.2.2 of \cite{Pie1}. For instance, to verify that $\text{Cov}^w_{\alpha X}$ satisfies condition (A8) from Definition 2.2.1 of \cite{Pie1}, when $\text{Ex}_{\alpha X}$ is admissibly additive, take a collection $\mathcal{U}\in \text{Cov}^w_{\alpha X}$ and any set $W\subseteq\bigcup\mathcal{U}$, such that $W\cap U\in\text{Op}^w_{\alpha X}$ whenever $U\in\mathcal{U}$. Then, for each $U\in\mathcal{U}$, we have $W\cap U\cap X\in\text{Op}_X$. Moreover, $X\cap_1\mathcal{U}\in\text{Cov}_X$, so $W\cap X\in\text{Op}_X$ and $(W\cap X)\cap_1\mathcal{U}\in\text{Cov}_X$. Since $\text{Ex}_{\alpha X}(W\cap U\cap X)=W\cap U$ for each $U\in\mathcal{U}$, we have $W=\bigcup_{U\in\mathcal{U}}(W\cap U)\in\text{Op}^{S}_{\alpha X}$. Furthermore, $\text{Ex}_{\alpha X}(W\cap X)=\text{Ex}_{\alpha X}[\bigcup_{U\in\mathcal{U}}(W\cap U\cap X)]=\bigcup_{U\in\mathcal{U}}\text{Ex}_{\alpha X}(W\cap U\cap X)=\bigcup_{U\in\mathcal{U}}(W\cap U)= W$.  
\end{proof}

\begin{rem}
Let $\mathcal{C}$ be a Wallman base of a semi-normal space $X$. The mapping $w_{\mathcal{C}}:X\to \mathcal{W}(X, \mathcal{C})$, defined by $w_{\mathcal{C}}(x)=\{A\in\mathcal{C}: x\in A\}$ for $x\in X$, is the most commonly used homeomorphic embedding of $X$ into $\mathcal{W}(X, \mathcal{C})$. By Theorem  2.8, the pair $(\mathcal{W}(X, \mathcal{C}), w_{\mathcal{C}})$ is a Hausdorff compactification of $X$ in \textbf{ZFU} but it is not necessarily a compactification of $X$ in a model for \textbf{ZF}.
\end{rem} 

 \begin{prop} 
Suppose that $X$ is a weakly normal gts and that the Wallman space $\mathcal{W}(X, \text{Cl}_X)$ is compact. Then, for $\mathcal{C}=\text{Cl}_X$, the collection $\text{Op}^{w}_{\mathcal{W}(X,\mathcal{C})}$ is stable under finite unions and generates the topology of $\mathcal{W}(X, \mathcal{C})$.
 \end{prop}
 \begin{proof}
It is known and easy to observe that, for arbitrary $A,B\in\mathcal{C}$, the equality $[A]_{\mathcal{C}}\cap [B]_{\mathcal{C}}= [A\cap B]_{\mathcal{C}}$ holds (cf. e.g. \cite{PW}). Let $U\in\text{Op}_X$. Then $X\setminus U\in\mathcal{C}$, so $\text{cl}_{\mathcal{W}(X, \mathcal{C})}(X\setminus U)=[X\setminus U]_{\mathcal{C}}$. This implies that $\text{Ex}_{\mathcal{W}(X, \mathcal{C})}(U)=\mathcal{W}(X, \mathcal{C})\setminus [X\setminus U]_{\mathcal{C}}$. Therefore, $\text{Op}^{w}_{\mathcal{W}(X, \mathcal{C})}=\{\mathcal{W}(X, \mathcal{C})\setminus [A]_{\mathcal{C}}: A\in\mathcal{C}\}$ and since $\{[A]_{\mathcal{C}}: A\in\mathcal{C}\}$ is a stable under finite intersections closed base for $\mathcal{W}(X,\mathcal{C})$, we deduce that $\text{Op}^{w}_{\mathcal{W}(X,\mathcal{C})}$ is stable under finite unions and it generates the topology of $\mathcal{W}(X,\mathcal{C})$.
 \end{proof} 

One can easily check that Taimanov's Theorem 3.2.1 of \cite{En} is valid in \textbf{ZF}. Therefore, we can say that it is proved independently in \cite{Bl} and \cite {W1} that the following useful modification of Taimanov's theorem also holds in \textbf{ZF}.

\begin{thm}
Let $X$ be a dense subspace of a topological space $T$ and let $f$ be a continuous mapping of $X$ into a compact Hausdorff space $Y$. Then $f$ is continuously extendable to a mapping $\tilde{f}:T\to Y$ if and only if there exists a closed base $\mathcal{F}$ of $Y$ such that $\mathcal{F}$ is stable under finite intersections and has the property that $\text{cl}_{T}f^{-1}(A)\cap\text{cl}_{T}f^{-1}(B)=\emptyset$ for each pair $A,B$ of disjoint members of $\mathcal{F}$. 
\end{thm} 

\begin{thm} Suppose we are given a gts $X$ and a Hausdorff topological compactification $\alpha X$ of $X_{top}$; moreover, let $\mathcal{C}=\text{Cl}_X$. Then the following conditions are equivalent:
\begin{enumerate}
\item[(i)] $\text{Ex}_{\alpha X}$ is finitely additive and $\text{Op}^{w}_{\alpha X}$ is an open base for a Hausdorff topology in $\alpha X$;
\item[(ii)] $X$ is weakly normal, the Wallman space $\mc{W}(X,\mc{C})$ is compact and, moreover, the compactification $\mathcal{W}(X,\mathcal{C})$ of $X_{top}$ is topologically equivalent with $\alpha X$.
\end{enumerate}
\end{thm} 
\begin{proof} Let $\tau$ be the original topology of $\alpha X$ and let $\tau^{w}$ be the topology in $\alpha X$ generated by $\text{Op}^{w}_{\alpha X}$. Assume that (\textit{i}) holds. Then $\tau=\tau^{w}$ by Corollary 3.1.14 of \cite{En}. Furthermore, the collection $\mathcal{C}^{w}=\{A\subseteq \alpha X: \alpha X\setminus A\in \text{Op}^{w}_{\alpha X}\}$ is a complete closed base for $\alpha X$. Since $\alpha X$ is a compact Hausdorff space, the collection $\mathcal{C}^{w}$ is a Wallman base for $\alpha X$. That $\text{Op}^{w}_{\alpha X}$ is stable under finite unions implies that, for an arbitrary pair $A, B$ of members of $\mathcal{C}$, the equality $\text{cl}_{\alpha X}(A)\cap\text{cl}_{\alpha X}(B)=\text{cl}_{\alpha X}(A\cap B)$ holds. This, together with the compactness of the Hausdorff space $\alpha X$, gives that $\mathcal{C}$ is a Wallman base of $X_{top}$. Therefore, $X$ is weakly normal by Proposition 1.8. To prove that $\mc{W}(X, \mc{C})$ is compact, let us consider any filter $\mc{F}$ in $\mc{C}$. Since $\alpha X$ is compact, there exists $p\in\bigcap_{A\in\mc{F}}\text{cl}_{\alpha X}A$. Let $\mc{M}=\{ A\in\mc{C}: p\in\text{cl}_{\alpha X}A\}$. That $\text{Ex}_{\alpha X}$ is finitely additive implies that $\mc{M}$ is a filter in $\mc{C}$. Suppose that $B\in\mc{C}$ and $p\notin\text{cl}_{\alpha X}B$. Since $\mathcal{C}^{w}$ is a complete closed base for $\alpha X$, there exists $D\in\mc{C}$ such that $p\in\text{cl}_{\alpha X}D\subseteq\alpha X\setminus \text{cl}_{\alpha X} B$. Then $D\in\mc{M}$ and $D\cap B=\emptyset$. Hence, $B$ does not belong to any filter in $\mc{C}$ that contains $\mc{M}$; thus $\mc{M}$ is a maximal filter in $\mc{C}$. We have shown that every filter in $\mc{C}$ can be enlarged to a maximal filter in $\mc{C}$. This implies that the Wallman space $\mc{W}(X, \mc{C})$ is compact.
Therefore, $\mathcal{W}(X, {\mathcal{C}})$ is a Hausdorff compactification of $X_{top}$. Applying Theorem 5.15, we obtain that the mapping $\alpha$ is continuously extendable to a mapping from $\mathcal{W}(X, \mathcal{C})$ onto $\alpha X$, while $w_{\mathcal{C}}$ is continuously extendable to a mapping from $\alpha X$ onto $\mathcal{W}(X, \mathcal{C})$; thus $\alpha X=\mathcal{W}(X, \mathcal{C})$. Therefore, (\textit{ii}) follows from (\textit{i}). Now, suppose that (\textit{ii}) is satisfied. The topological equivalence $\alpha X=\mathcal{W}(X, \mathcal{C})$ implies that $\text{Op}^{w}_{\alpha X}$ is an open base for $\tau$. In addition,  for an arbitrary pair $A, B$ of sets from $\mathcal{C}$, we have $\text{cl}_{\alpha X}(A)\cap\text{cl}_{\alpha X}(B)=\text{cl}_{\alpha X}(A\cap B)$ and, in consequence,  $\text{Ex}_{\alpha X}$ is finitely additive.
\end{proof} 

When $\text{Ex}_{\alpha X}$ is finitely additive and, simultaneously, $\text{Op}^{w}_{\alpha X}$ generates the topology of $\alpha X$, that $\text{Cov}^{w}_{\alpha X}$ need not be a generalized topology in $\alpha X$ is shown by the following example: 

\begin{exam} 
(\textbf{ZFU}) For the usual topological real line $\mathbb{R}_{ut}$ (cf. Definition 1.2 (i)), put  $\mathcal{C}=\{A\subseteq \mathbb{R}: \mathbb{R}\setminus A\in \tau_{nat}\}$.  Let $\beta\mathbb{R}=\mathcal{W}(\mathbb{R},\mathcal{C})$. Then, by Theorem 2.8,  it is true in \textbf{ZFU} that $\beta\mathbb{R}$ is a Hausdorff compactification of $\mathbb{R}$; of course, it is valid in \textbf{ZFU} that $\beta\mathbb{R}$ is the greatest compactification of $\mathbb{R}$, i.e the \v{C}ech-Stone compactification of $\mathbb{R}$. Observe that  $\text{Ex}_{\beta\mathbb{R}}[(-n; n)]=(-n; n)$ for each $n\in\omega$. Moreover, $\text{Ex}_{\beta\mathbb{R}}[\bigcup_{n\in\omega}(-n; n)]=\beta\mathbb{R}$. Therefore, $\text{Ex}_{\beta\mathbb{R}}$ is not admissibly additive, so, by Proposition 5.12, the collection $\text{Cov}^{w}_{\beta\mathbb{R}}$ is not a generalized topology. By Proposition 5.14,  $\text{Op}^w_{\beta\mathbb{R}}$ is stable under finite unions and generates the topology of $\beta\mathbb{R}$. Let us observe that the space $\mathcal{W}(\mathbb{R}, \mathcal{C})$ is non-compact in Solovay's model $\mathcal{M}5(\aleph )$ of \cite{HR}. 
\end{exam}

\begin{defi} Let $X$ be a weakly normal gts such that, for $\mathcal{C}=\text{Cl}_X$, the Wallman space $\mathcal{W}(X,\mathcal{C})$ is compact. If the wallmanian part $\text{Cov}^{w}_{\mathcal{W}(X,\mathcal{C})}$ is a generalized topology in $\mathcal{W}(X, \mathcal{C})$, then the gts $$wX=(\mathcal{W}(X,\mathcal{C}), \text{Op}^{w}_{\mathcal{W}(X, \mathcal{C})}, \text{Cov}^{w}_{\mathcal{W}(X, \mathcal{C})})$$ will be called \textbf{the Wallman strict compactification} of $X$.
\end{defi}

The theorem below follows from Theorems 2.8 and 5.16 taken together with Propositions 5.12 and  5.14:
\begin{thm} 
(\textbf{ZFU}) A weakly normal gts $X$ has its strict Wallman compactification if and only if the operator $\text{Ex}_{\mathcal{W}(X, \text{Cl}_X)}$ is admissibly additive. 
\end{thm}

\begin{f} Suppose that $X$ is a small gts and that $\alpha X$ is a Hausdorff compactification of the topologization of $X$. If $\text{Op}^{w}_{\alpha X}$ is stable under finite unions, then $\text{Cov}^{w}_{\alpha X}=\text{EssFin}(\text{Op}^{w}_{\alpha X})$ is a generalized topology in $\alpha X$.
\end{f}

\begin{thm} 
For every small gts $X$ and for every  Hausdorff topological compactification $\alpha X$ of $X_{top}$, the following conditions are equivalent:
\begin{enumerate}
\item[(i)] $\text{Cov}^{w}_{\alpha X}$ is a generalized weakly Hausdorff topology in $\alpha X$;
\item[(ii)] $X$ is weakly normal and, simultaneously, the Wallman space $\mathcal{W}(X, \text{Cl}_X)$ of $X_{top}$ is a compactification of $X_{top}$ topologically equivalent with $\alpha X$. 
\end{enumerate}
\end{thm}
\begin{proof} It suffices to apply Proposition 5.12, Theorem 5.16 and Fact 5.20.
\end{proof} 

\begin{exam} (\textbf{ZFU}) Let us find out which ones of the real lines from Definition 1.2 have their Wallman strict compactifications. Clearly, the Wallman space of the partially topological real lines $\mb{R}_{ut}, \mb{R}_{st}, \mb{R}_{lst}, \mb{R}_{l^{+}st}$ and $\mb{R}_{l^{-}st}$ is $\beta\mb{R}$. Example 5.17 shows $\mb{R}_{ut}, \mb{R}_{lst}, \mb{R}_{l^{+}st}$ and $\mb{R}_{l^{-}st}$ do not have their Wallman strict compactifications, while, in view of Theorems 5.19 and 5.21, $\mb{R}_{st}$ has its Wallman strict compactification. The Wallman spaces of $\mb{R}_{om}$ and $\mb{R}_{rom}$ are  topologically equivalent with the two-point Hausdorff compactification of $(\mb{R}, \tau_{nat})$ and, clearly, the gtses $\mb{R}_{om}$ and $\mb{R}_{rom}$  have their Wallman strict compactifications. It is not hard to check that the Wallman spaces of $\mb{R}_{lom}$ and $\mb{R}_{slom}$ are topologically equivalent with $\beta\mb{R}$ and, in view of Example 5.17, the gtses $\mb{R}_{lom}$ and $\mb{R}_{slom}$ do not have their Wallman strict compactifications. Using similar arguments, one can deduce that $\mb{R}_{l^{+}om}, \mb{R}_{l^{-}om}, \mb{R}_{sl^{+}om}$ and $\mb{R}_{sl^{-}om}$ do not have their Wallman strict compactifications. 
\end{exam} 

Finally, as a consequence of Theorems 2.8 and 5.21, let us write down the following theorem:
\begin{thm}
That every weakly normal small gts has its Wallman strict compactification is an equivalent of \textbf{UFT}. 
\end{thm}

\section{A construction of a strict compactification with a given remainder}

For gtses $X$ and $K$, it is interesting to know when there exists a strict compactification $\alpha X$ of $X$ such that $K=\alpha X\setminus X$ and $K$ is a strict subspace of $\alpha X$; moreover,  if such a strict compactification exists, it is worthwhile to describe a comfortable method of constructing it. Similarly to the classical theory of compactifications, when the gts $X$ is weakly normal, we wish to obtain weakly Hausdorff strict compactifications of $X$.  
Let us consider a relatively simple case when $X$ is a weakly normal but not topologically compact gts and $K$ is a weakly Hausdorff topologically compact non-void gts.  If there exists a strict weakly Hausdorff compactification $\alpha X$ of $X$ such that $\alpha X\setminus X=K$, then $X_{top}$ is locally compact. Therefore, let us assume that $X_{top}$ is locally compact. By Proposition 3.15, the space $K$ is weakly normal. Thus, by Proposition 1.8, the collection $\text{Cl}_K$ is a  Wallman base of $K_{top}$. Clearly, $\text{Cl}_X$ is a Wallman base of $X_{top}$. Let us mimick the results of \cite{W4} and denote by $\sim $ the equivalence relation on $\text{Cl}_X$ such that, for an arbitrary pair $A, B$ of members of $\text{Cl}_X$, we have $A\sim B$ if and only if the symmetric difference $A\bigtriangleup B$ does not contain topologically non-compact members of $\text{Cl}_X$. For $A\in\text{Cl}_X$, let $[A]^X$ denote the equivalence class of $\sim$ such that $A\in[A]^X$.  The quotient set $L(\text{Cl}_X)=\{ [A]^X: A\in\text{Cl}_X\}$ is a distributive lattice if we put $[A]^X\wedge [B]^X= [A\cap B]^X$ and $[A]^X\vee  [B]^X=[A\cup B]^X$ for an arbitrary pair $A, B$ of members of $\text{Cl}_X$. In much the same way, as in the proof to Theorem 2.8, we can show that it is valid in \textbf{ZFU} that every filter in $L(\text{Cl}_X)$ is contained in an ultrafilter in $L(\text{Cl}_X)$. Let $S(\text{Cl}_X)$ be the set of all ultrafilters in $L(\text{Cl}_X)$. For $A\in\text{Cl}_X$, let $H[A]=\{a\in S(\text{Cl}_X): [A]^X\in a\}$. We consider $S(\text{Cl}_X)$ with the smallest topology having $\{ H[A]: A\in\text{Cl}_X\}$ as a base for closed sets. Although the work \cite{W4} was in \textbf{ZFC}, we can state the following relevant result: 

\begin{prop} 
(\textbf{ZFU}) The collection $\{ H[A]: A\in\text{Cl}_X\}$ is a closed base for a Hausdorff compact topology in $S(\text{Cl}_X)$.
\end{prop}

Without any loss of generality, we may assume that the sets $X$ and $K$ are disjoint  for, if they are not disjoint, we can replace $X$ by $X\times\{0\}$ and $K$ by $K\times\{1\}$. The disjoint union of the sets $X$ and $K$ is  $(X\times\{0\})\cup(K\times\{ 1\})$ when $X\cap K\neq\emptyset$. The disjoint union of $X$ and $K$ is the set $X\cup K$ when $X\cap K=\emptyset$. 
 
\begin{defi} Suppose that, for a weakly normal gts $X$ and for a weakly Hausdorff topologically compact gts $K$,  we are given a lattice isomorphism $\psi:\text{Cl}_K\to L(\text{Cl}_X)$. Then
\begin{enumerate} 
\item[(i)] $X\cup_{\psi} K$ is the disjoint union of $X$ and $K$;
\item[(ii)] $\text{Cl}^w_{X\cup_{\psi} K}=\{ A\cup\psi^{-1}([A]^X): A\in\text{Cl}_X)\}$;
\item[(iii)] $\text{Op}^w_{X\cup_{\psi}K}=\{ U\subseteq X\cup_{\psi}K: (X\cup_{\psi}K)\setminus U\in\text{Cl}^w_{X\cup_{\psi} K}\}$;
\item[(iv)] $\tau_{X\cup_{\psi}K}= \tau(\text{Op}^w_{X\cup_{\psi}K})$;
\item[(v)] $(X\cup_{\psi}K)_{top}=(X\cup_{\psi}K, \tau_{X\cup_{\psi}K})$;
\item[(vi)] $\text{Op}^{s}_{X\cup_{\psi}K}=\{ U\in\tau_{X\cup_{\psi}K}: U\cap X\in\text{Op}_X \wedge U\cap K\in\text{Op}_K\}$;
\item[(vii)] $\text{Cov}^{s}_{X\cup_{\psi}K}=\{ \mathcal{U}\subseteq \text{Op}^{s}_{X\cup_{\psi}K}: X\cap_1\mathcal{U}\in\text{Cov}_X \wedge K\cap_1\mathcal{U}\in\text{Cov}_K\}$;
\item[(viii)] $\alpha^{s}_{\psi} X=( X\cup_{\psi}K, \text{Op}^{s}_{X\cup_{\psi}K}, \text{Cov}^{s}_{X\cup_{\psi}K})$;
\item[(ix)] we say that $\text{Cov}_X$ and $\text{Cov}_K$ are $\psi$-correlated if $$\text{Cov}_K=\{\{ K\setminus \psi^{-1}([X\setminus U]^X): U\in\mc{U}\}: \mc{U}\in\text{Cov}_X\}.$$
\end{enumerate}
\end{defi}

\begin{thm} 
Suppose we are given a weakly normal, topologically  both non-compact and locally compact gts $X$, a weakly Hausdorff topologically compact gts $K$ and a lattice isomorphism $\psi:\text{Cl}_K\to L(\text{Cl}_X)$. Then $\alpha^{s}_{\psi}X$ is a weakly Hausdorff strict compactification of $X$ such that the gts $K$ is a strict subspace of $\alpha^{s}_{\psi}X$ if at least one of the following conditions (i)--(ii) is satisfied:
\begin{enumerate}
\item[(i)] $\text{Cov}_X$ and $\text{Cov}_K$ are $\psi$-correlated,
\item[(ii)] both $X$ and $K$ are small.
\end{enumerate}
\end{thm}
\begin{proof} Mimicking the proof to Theorem 2.1 of \cite{W4}, one can show that $(X\cup_{\psi}K)_{top}$ is a Hausdorff compact space and that $X$ is dense in $(X\cup_{\psi}K)_{top}$. Using simple arguments, one can verify that $\text{Cov}^{s}_{X\cup_{\psi}K}$ is a generalized topology in the set $X\cup_{\psi} K$ such that $\text{Op}^{w}_{X\cup_\psi K}\subseteq\text{Op}^{s}_{X\cup_{\psi} K}=\bigcup\text{Cov}^{s}_{X\cup_{\psi}K}$. It is easily seen that $\langle X\cap_2 \text{Cov}^{s}_{X\cup_{\psi}K}\rangle_X\subseteq \text{Cov}_X$ and $\langle K\cap_2\text{Cov}^{s}_{X\cup_{\psi}K}\rangle_K\subseteq \text{Cov}_K$. 

Assume that (i) holds. Consider arbitrary $\mc{V}\in\text{Cov}_K$ and  $\mc{U}\in\text{Cov}_X$ such that $\mc{V}=\{K\setminus\psi^{-1}([X\setminus U]^X): U\in\mc{U}\}$. It is obvious that $\mc{W}=\{U\cup(K\setminus \psi^{-1}([X\setminus U]^{X})): U\in\mc{U}\}\in \text{Cov}^{s}_{X\cup_{\psi}K}$. Moreover,  $\mc{W}\cap_1 K=\mc{V}$ and $\mc{W}\cap_1 X=\mc{U}$. This is why $\mc{V}\in K\cap_2\text{Cov}^{s}_{X\cup_{\psi}K}$ and $\mc{U}\in X\cap_2 \text{Cov}^{s}_{X\cup_{\psi}K}$. In consequence, $K\cap_2\text{Cov}^{s}_{X\cup_{\psi}K}=\text{Cov}_K$ and $X\cap_2 \text{Cov}^{s}_{X\cup_{\psi}K}=\text{Cov}_X$.  

Finally, assume that $X$ and $K$ are small. Then $\langle X\cap_2 \text{Cov}^{s}_{X\cup_{\psi}K}\rangle_X=\text{Cov}_X$ and $\langle K\cap_2\text{Cov}^{s}_{X\cup_{\psi}K}\rangle_K=\text{Cov}_K$, so both $X$ and $K$ are subspaces of $\alpha^{s}_{\psi} X$. Now, that $X$ and $K$ are strict subspaces of $\alpha^{s}_{\psi} X$ follows from Proposition 2.3.10 of \cite{Pie1}.
\end{proof}

\begin{rem} Under the assumptions of Theorem 6.3, we have $$X\in\text{Op}^{s}_{X\cup_{\psi}K}\setminus \text{Op}^{w}_{X\cup_{\psi}K}.$$  
 \end{rem}
 
\begin{prop} 
Suppose that $X$ is a weakly normal, topologically both non-compact and locally compact gts,  $K$ is a weakly Hausdorff topologically compact gts, while $\psi:\text{Cl}_K\to L(\text{Cl}_X)$ is a lattice isomorphism. Then the following conditions are satisfied:
\begin{enumerate}
\item[(i)] $(X\cup_{\psi}K)_{top}$ is a Hausdorff compactification of $X_{top}$;
\item[(ii)] $\text{Cl}^w_{X\cup_{\psi}K}$ is a Wallman base for $(X\cup_{\psi}K)_{top}$;
\item[(iii)] for each $A\in\text{Cl}_X$, the equality holds: $\text{cl}_{(X\cup_{\psi}K)_{top}}(A)=A\cup\psi^{-1}([A]^X)$;
\item[(iv)] the Wallman space $\mc{W}(X, \text{Cl}_X)$ is a compactification of $X_{top}$ topologically equivalent with $(X\cup_{\psi}K)_{top}$;
\item[(v)] the space $S(\text{Cl}_X)$ is a compact Hausdorff space homeomorphic with $K_{top}=(X\cup_{\psi}K)_{top}\setminus X_{top}$.
\end{enumerate}
\end{prop}
\begin{proof} That (i)-(iii) hold can be deduced from the proof to Theorem 2.1 of \cite{W4}. That every filter in $\text{Cl}_X$ is contained in an ultrafilter in $\text{Cl}_X$ can be proved in much the same way, as that  condition (i) of Theorem 5.16 implies the compactness of $\mc{W}(X,\text{Cl}_X)$. Using similar arguments to the ones from the proof of Theorem 2.1 in \cite{W4}, one can show that the compactifications $\mc{W}(X, \text{Cl}_X)$ and  $(X\cup_{\psi}K)_{top}$ of $X_{top}$ are topologically equivalent. Since every filter in $\text{Cl}_X$ is contained in an ultrafilter in $\text{Cl}_X$, the space $S(\text{Cl}_X)$ is also compact. That it is Hausdorff, is shown in Proposition 1.4 of \cite{W4}. Finally, by applying the arguments of the proof to Theorem 2.2 given in \cite{W4}, we obtain that (v) holds.
\end{proof}

We deduce from the proposition above, taken together with Theorem 5.16 and 5.21, that the following theorem holds:

\begin{thm} 
Let $X$ be a weakly normal, topologically both non-compact and locally compact small gts. Suppose that $K$ is a weakly Hausdorff, topologically compact small gts such that there exists a lattice isomorphism $\psi:\text{Cl}_K\to L(\text{Cl}_X)$. Then   
$$\alpha^{w}_{\psi} X=(X\cup_{\psi}K, \text{Op}^w_{X\cup_{\psi}K}, \text{EssFin}(\text{Op}^w_{X\cup_{\psi}K}))$$ 
is a strict compactification of $X$ such that $\alpha^{w}_{\psi} X$ is strictly equivalent with the Wallman strict compactification of $X$, while $K$ is a strict subspace of ${\alpha}^{w}_{\psi}X$.
\end{thm}

\begin{rem} Under the assumptions of Theorem 6.6, we have $\alpha^{w}_{\psi}X\preceq \alpha^{s}_{\psi}X$; however, it follows from Remark 6.4 that the strict compactifications $\alpha^{s}_{\psi}X$ and $\alpha^{w}_{\psi}X$ of $X$ are not strictly equivalent.
\end{rem}

\section{Strict compactifications with finite remainders}

Let us recall useful facts  concerning Alexandroff's compactifications of topological locally compact, non-compact Hausdorff spaces.

\begin{f} Every non-compact, locally compact Hausdorff topological space $X$ has the following properties:
\begin{enumerate}
\item[(i)] there exists exactly one (up to equivalence) Hausdorff compactification $\alpha_{0}X$ of $X$ such that $\alpha_{0}X\setminus X$ is a singleton;
\item[(ii)] if $\mathcal{C}_0$ is the collection of all closed sets $A$ in $X$ such that either $A$ is compact or every closed in $X$ set $B$ disjoint from $A$ is compact, then $\mathcal{C}_0$ is a Wallman base for $X$ such that the Wallman space $\mathcal{W}(X,\mathcal{C}_{0})$ is a Hausdorff compactification of $X$ equivalent to $\alpha_{0}X$.
\end{enumerate}
\end{f} 
\begin{proof} 
That (\textit{ii}) holds can be deduced from the results of Section 2 of \cite{StSt} (cf. e.g. Theorem 5 in \cite{StSt}) or from our Proposition 7.14 given below. 
\end{proof}
\begin{f} Every locally compact Hausdorff topological space is a regular semi-normal space. 
\end{f}
\begin{rem} 
It seems unknown whether a locally compact Hausdorff topological space can be not completely regular in a model for \textbf{ZF}.
\end{rem}
 
\begin{defi}
 A one-point strict compactification of a topologically non-compact gts $X$ is a strict compactification $\hat{X}$ of $X$ such that the set $\hat{X}\setminus X$ is a singleton. 
\end{defi}

For a gts $X$, let $\text{Kc}_X$ be the collection of all topologically compact sets from $\text{Cl}_X$.
\begin{defi}
Assume that $X$ is a gts which is topologically non-compact. For an element $\infty\notin X$, let: $$\hat{X}=X\cup\{\infty\},$$
$$\text{Op}_{\hat{X}}=\text{Op}_X\cup \{ \hat{X}\setminus C: C\in\text{Kc}_X\},$$
$$\text{Cov}_{\hat{X}}= \{ \mathcal{U}\subseteq \text{Op}_{\hat{X}}: \mathcal{U}\cap_1 X
\in \text{Cov}_X \}.$$
Then $\hat{X}=(\hat{X}, \text{Op}_{\hat{X}}, \text{Cov}_{\hat{X}})$ is called \textbf{the Alexandroff strict compactification} of $X$.
 \end{defi}
 We hope that readers will explain to themselves the following facts:

\begin{f} Assume that $X$ is a gts which is topologically non-compact and $\hat{X}$ is the Alexandroff strict compactification of $X$. Then:
\begin{enumerate} 
\item[(i)] $\hat{X}$  is a strict one-point compactification of $X$.
\item[(ii)] If $\hat{X}$ is weakly Hausdorff, then $\text{Cov}_{\hat{X}}$ is the strongest generalized topology in $\hat{X}_{top}$ associated with $\text{Cov}_X$.
\item[(iii)] If $\hat{X}$ is weakly Hausdorff, then $\hat{X}$ is the unique up to strict equivalence one-point weakly Hausdorff strict compactification of $X$ equipped with the strongest generalized topology associated with $\text{Cov}_X$.
\end{enumerate}
\end{f}

The following example shows that a weakly Hausdorff one-point compactification of a topologically non-compact locally compact gts $X$ need not be strictly equivalent with $\hat{X}$. 
\begin{exam}
Let $Y$ be the one-point Hausdorff topological compactification of the real line $\mb{R}$ equipped with its natural topology. Let $\mc{U}$ be the collection of all simultaneously open and bounded intervals in $\mb{R}$ and let $\mc{V}=\{\emptyset, Y\}\cup \mc{U}\cup\{ Y\setminus \text{cl}_{\mb{R}}U: U\in \mc{U}\}$. Define $\text{Op}_Y$ as the collection which consists of all finite unions of members of $\mc{V}$. Let $\text{Cov}_Y=\text{EssFin}(\text{Op}_Y)$. Then $(Y, \text{Cov}_Y)$ is a one-point strict compactification of its subspace $(\mb{R}, \mb{R}\cap_2\text{Cov}_Y)$. However, the Alexandorff strict compactification of $(\mb{R}, \mb{R}\cap_2\text{Cov}_Y)$ is not strictly equivalent with $(Y, \text{Cov}_Y)$ because $\mb{R}\notin\text{Op}_Y$.
\end{exam}

It is interesting to have a look at the Alexandroff strict compactifications of $\mathbb{R}^{n}$ equipped with distinct generalized topologies that induce the natural topology of $\mathbb{R}^{n}$. 

\begin{exam}\label{sfera}
The small partially topological space $(\mathbb{R}^n, \tau, \text{EssFin}(\tau))$ where $\tau$ is the natural topology of $\mathbb{R}^n$ (cf. Definition 2.2.14(7) of \cite{Pie1}), has as its one-point strict compactification the small partially topological sphere $S^n$ (with the natural topology).   
\end{exam}
\begin{exam}
For $i\in\mathbb{N}$, let $B_i$ be the standard open ball in $\mathbb{R}^n$ at centre $0$ and of radius $i$.  The generalized topology $\langle \{\{B_i: i\in\mathbb{N}\}\}\cup\text{EssFin}(\tau)\rangle$ of the localization $\mb{R}^n_{loc}$  (cf. Definition 2.1.15 of \cite{Pie2}) of the Euclidean space from Example \ref{sfera} 
is only locally small but not small, and the Alexandroff strict compactification of $\mb{R}^n_{loc}$  is a (not locally small) sphere with the point "at infinity"  having no small neighbourhood.
\end{exam}

Example 5.7 shows that the Alexandroff strict compactification of a topologically locally compact, non-compact weakly Hausdorff gts can be not weakly Hausdorff. Let us go in search for necessary and sufficient conditions for a gts to have a weakly Hausdorff one-point strict compactification. 

\begin{lem} Let $X$ be a weakly normal gts such that $X_{top}$ is locally compact. Suppose that $V\in\text{Op}_X$ and $x\in V$. Then there exist $U\in\text{Op}_X$ and $C\in\text{Kc}_X$ such that $x\in U\subseteq C\subseteq V$.
\end{lem}
\begin{proof} Since $X_{top}$ is a locally compact Hausdorff space, it is regular. Therefore, there exists $W\in\text{Op}_X$ such that $x\in W\subseteq \text{cl}_{X_{top}}(W)\subseteq V$ and the set $K=\text{cl}_{X_{top}}(W)$ is topologically compact. Let $B=X\setminus W$.  There exits $A\in\text{Cl}_X$ such that $x\in A\subseteq W$. Take sets $C, D\in\text{Cl}_X$ such that $A\subseteq X\setminus D$, $B\subseteq X\setminus C$ and $C\cup D=X$. Let $U=X\setminus D$. Then $x\in U\subseteq C\subseteq W$ and, moreover,  $C\in\text{Kc}_X$ because $C$ is a closed subset of a topologically compact set $K$.
\end{proof}

It suffices to apply the lemma given above to infer that the following proposition holds:

\begin{prop} 
A weakly normal topologically non-compact gts $X$ has a weakly Hausdorff one-point strict compactification if and only if $X_{top}$ is locally compact.
\end{prop}

\begin{prop}
Let $\hat{X}$ be the Alexandroff strict compactification of a weakly normal gts $X$ such that $X_{top}$ is a locally compact, non-compact space. Then the operator $\text{Ex}_{\hat{X}}$ is finitely additive if and only if the following condition is fulfilled: for every pair $A, B$ of disjoint members of $\text{Cl}_X$, at least one of the sets $A, B$ is topologically compact. 
\end{prop}
\begin{proof} Let us observe that, for each $U\in\text{Op}_X$, we have $\text{Ex}_{\hat{X}}(U)=U$ if $X\setminus U\notin\text{Kc}_X$, while $\text{Ex}_{\hat{X}}(U)=U\cup\{\infty\}$ when $X\setminus U\in\text{Kc}_X$. For $A\in\text{Cl}_X$, let $U_A=X\setminus A$. 

\textsl{Necessity}. First, assume that $\text{Ex}_{\hat{X}}$ is finitely additive. Consider an arbitrary pair $A, B$ of disjoint members of $\text{Cl}_X$. Since $U_{A}\cup U_{B}=U_{A\cap B}=X$, we have $\hat{X}=\text{Ex}_{\hat{X}}(U_{A}\cup U_{B})=\text{Ex}_{\hat{X}}(U_A)\cup\text{Ex}_{\hat{X}}(U_B)$. Then $\infty\in\text{Ex}_{\hat{X}}(U_A)$ or $\infty\in\text{Ex}_{\hat{X}}(U_B)$. This implies that $A\in\text{Kc}_X$ or $B\in\text{Kc}_X$. 

\textsl{Sufficiency.} Now, suppose that, for an arbitrary pair $A, B$ of disjoint members of $\text{Cl}_X$, at least one of the sets $A, B$ is in $\text{Kc}_X$. To show that $\text{Ex}_{\hat{X}}$ is finitely additive, take any pair $U, V$ of sets from $\text{Op}_X$. Let $A=X\setminus U$ and $B=X\setminus V$. If $A\cap B\notin\text{Kc}_X$, then it is obvious that $\text{Ex}_{\hat{X}}(U\cup V)= U\cup V=\text{Ex}_{\hat{X}}(U)\cup\text{Ex}_{\hat{X}}(V)$. Suppose that $A\cap B\in\text{Kc}_X$. It follows from Lemma 7.10 that there exist sets $C\in\text{Kc}_X$ and $W\in\text{Op}_X$, such that $A\cap B\subseteq W\subseteq C$. If $A\notin\text{Kc}_X$, then $B\setminus W\in\text{Kc}_X$ since $B\setminus W\subseteq X\setminus A$. In consequence,  $B=(B\cap C)\cup (B\setminus W)\in\text{Kc}_X$. Similarly, if $B\notin\text{Kc}_X$, then $A\in\text{Kc}_X$. Therefore, $\infty\in\text{Ex}_{\hat{X}}(U)\cup\text{Ex}_{\hat{X}}(V)$ and this concludes the proof because $\infty\in\text{Ex}_{\hat{X}}(U\cup V)$ when $A\cap B\in\text{Kc}_X$. 
\end{proof}

\begin{thm} Suppose that $X$ is a weakly normal gts such that the topological space $X_{top}$ is locally compact but not compact. If the Alexandroff strict compactification $\hat{X}$ of $X$ is such that $\text{Ex}_{\hat{X}}$ is finitely additive, then the Wallman space $\mathcal{W}(X, \text{Cl}_X)$ is compact and, simultaneously, the compactifications $\hat{X}_{top}$ and $\mathcal{W}(X, \text{Cl}_X)$ of $X_{top}$ are topologically equivalent.
\end{thm}
\begin{proof} Assume that $\text{Ex}_{\hat{X}}$ is finitely additive. Let $\mathcal{F}$ be a filter in $\text{Cl}_X$. If $\mathcal{F}\cap\text{Kc}_X\neq\emptyset$, then $\mathcal{F}$ is contained in a fixed maximal filter in $\text{Cl}_X$. Suppose that $\mathcal{F}\cap\text{Kc}_X=\emptyset$. Let $\mathcal{M}=\text{Cl}_X\setminus\text{Kc}_X$. With Proposition 7.12 in hand, it is easy to check that $\mathcal{M}$ is a maximal filter in $\text{Cl}_X$ such that $\mathcal{F}\subseteq\mathcal{M}$. In consequence, since every filter in $\text{Cl}_X$ is contained in an ultrafilter in $\text{Cl}_X$, the Wallman space $\mathcal{W}(X, \text{Cl}_X)$ is compact. It is easily seen that $\tau(\text{Op}^{w}_{\hat{X}})$ coincides with $\tau(\text{Op}_{\hat{X}})$. Now, analysing the proof to Theorem 5.16, we infer that the compactifications $\mathcal{W}(X, \text{Cl}_X)$ and $\hat{X}_{top}$ of $X_{top}$ are topologically equivalent.
\end{proof}

Our next proposition, more general than Fact 7.1 (\textit{ii}), gives a method of finding weakly normal gtses $X$ such that the operators $\text{Ex}_{\hat{X}}$ are finitely additive and says more about the Alexandroff strict compactification of a topologically non-compact locally compact gts.  

\begin{prop} 
Let $X$ be a weakly normal gts such that $X_{top}$ is locally compact and topologically non-compact. Then the following properties hold:
\item[(i)] the collection
$$\mathcal{C}_0=\{ A\in\text{Cl}_X: A\in\text{Kc}_X \vee \forall_{B\in\text{Cl}_X}[B\subseteq X\setminus A\Rightarrow B\in\text{Kc}_X]\}$$
is a Wallman base for $X_{top}$ such that the Wallman space $\mc{W}(X, \mc{C}_0)$ is compact;
\item[(ii)] if $\mc{W}(X, \mc{C}_0)^S$ is $\mc{W}(X, \mc{C}_0)$ equipped with the strongest generalized topology associated with $\text{Cov}_X$, then the strict compactification $\mc{W}(X, \mc{C}_0)^S$ of $X$ is strictly equivalent with the Alexandroff strict compactification $\hat{X}$ of $X$.
\end{prop}

\begin{proof} 
First, we are going to prove that $\mc{C}_0$ is a Wallman base. Let $A_1,A_2\in\mathcal{C}_0$. We need to show that $A_1\cap A_2\in \mc{C}_0$. If $A_1\cap A_2\in\text{Kc}_X$, then $A_1\cap A_2\in\mathcal{C}_0$. Suppose that $A_1\cap A_2\notin\text{Kc}_X$. Take $B\in\text{Cl}_X$ such that $B\subseteq X\setminus (A_1\cap A_2)$. Then $B\cap A_1\subseteq X\setminus A_2$, so $B\cap A_1\in\text{Kc}_X$. Similarly, $B\cap A_2\in\text{Kc}_X$. It follows from Lemma 7.10 that there exist sets $V\in\text{Op}_X$ and $C\in\text{Kc}_X$, such that $B\cap(A_1\cup A_2)\subseteq V\subseteq C$. Then $B\setminus V\in\text{Cl}_X$ and $B\setminus V\subseteq X\setminus (A_1\cup A_2)$. Since $A_1\cup A_2\notin\text{Kc}_X$, then $B\setminus V\in\text{Kc}_X$ and, in consequence, $B=(B\setminus V)\cup(B\cap C)\in \text{Kc}_X$. Therefore, $A_1\cap A_2\in\mathcal{C}_0$. If $A_1\cup A_2\notin\text{Kc}_X$, $D\in\text{Cl}_X$ and $D\subseteq X\setminus (A_1\cup A_2)$, then it is obvious that $D\in\text{Kc}_X$. All this taken together implies that $\mathcal{C}_0$ is stable under finite unions and under finite intersections. If $H\in\text{Cl}_X$ and $x\notin H$, it follows from Lemma 7.10 that there exist $U_x\in\text{Op}_X$ and $H_x\in\text{Kc}_X$ such that $x\in U_x\subseteq H_x\subseteq X\setminus H$. This proves that the ring $\mathcal{C}_0$ is disjunctive. Furthermore, $X\setminus U_x\in\mathcal{C}_0$, $x\notin X\setminus U_x$ and $H\subseteq X\setminus U_x$. This proves that $\mathcal{C}_0$ is a closed base of $X_{top}$

Assume that the sets $A_1, A_2\in\mathcal{C}_0$ are disjoint. Suppose that one of the sets $A_1, A_2$ is topologically compact. Let $A_1$ be topologically compact. In view of Lemma 7.10, there exist $W_1\in\text{Op}_X$ and $F_1\in\text{Kc}_X$, such that $A_1\subseteq W_1\subseteq F_1\subseteq X\setminus A_2$. Put $F_2=X\setminus W_1$. Then $F_1, F_2\in\mathcal{C}_0$. $F_1\cup F_2=X$ and $A_1\subseteq X\setminus F_2$, while $A_2\subseteq X\setminus F_1$. 

Now, assume that none of the sets $A_1, A_2$ is topologically compact.  There exist $C_1, C_2\in\text{Cl}_X$ such that $C_1\cup C_2= X$, $A_1\subseteq X\setminus C_1$ and $A_2\subseteq X\setminus C_2$. Take any $E\in\text{Cl}_X$ such that $E\subseteq X\setminus C_1$. Then $E\subseteq C_2\subseteq X\setminus A_2$, so $E\in\text{Kc}_X$. In consequence, $C_1\in\mathcal{C}_0$. Similarly, $C_2\in\mathcal{C}_0$. Hence $\mc{C}_0$ is a Wallman base for $X_{top}$. The compactness of $\mc{W}(X, \mc{C}_0)$ follows from Proposition 7.12 and Theorem 7.13. That (ii) holds can be deduced from (i) taken together with Theorem 7.13 and fact 7.6 (iii).
\end{proof}

\begin{rem} 
Suppose we are given a weakly normal topologically non-compact gts $(X, \text{Op}_X, \text{Cov}_X)$ such that $X_{top}$ is locally compact. Let $\mathcal{C}_0$ be as in Proposition 7.14. Put $\text{Op}^{\ast}=\{U\subseteq X: X\setminus U\in\mathcal{C}_0\}$ and $\text{Cov}^{\ast}=\text{EssFin}(\text{Op}^{\ast})$. Notice that $X^{\ast}=(X,\text{Op}^{\ast}, \text{Cov}^{\ast})$ is a gts such that, when $\hat{X}^{\ast}$ is the Alexandroff strict compactification of $X^{\ast}$, then $\text{Ex}_{\hat{X}^{\ast}}$ is finitely additive. 
\end{rem}

Remarks 7.16 and 7.17 below are relevant to Magill's theorem on when a locally compact Hausdorff space can have a Hausdorff compactification with its remainder of cardinality $n\in\omega$ (cf. Theorem 2.1 \cite{Mag} or Theorem 6.8 of \cite{Ch}). 

\begin{rem} Suppose that $\alpha X$ is a weakly Hausdorff strict compactification of a gts $X$ such that $X_{top}$ is not compact. Assume that the set $\alpha X\setminus X$ is of cardinality $n\in\omega\setminus\{0\}$. Then there exists a collection $\{U_i: i\in n\}\subseteq\text{Op}_{X}$ such that $C=X\setminus\bigcup_{i\in n}U_i\in\text{Kc}_X$, $U_i\cap U_j=\emptyset$ for each pair $i, j$ of distinct elements of $n$,  and $U_i\cup C\notin\text{Kc}_X$ for each $i\in n$.
\end{rem}

\begin{rem}
Let $X$ be a weakly normal gts such that $X_{top}$ is locally compact. Assume that $n\in\omega\setminus\{0\}$ and that $\{U_i: i\in n\}\subseteq\text{Op}_{X}$ is a collection of pairwise disjoint sets such that $C=X\setminus\bigcup_{i\in n}U_i\in\text{Kc}_X$, $U_i\cup C\notin\text{Kc}_X$ for each $i\in n$ and, moreover,  every set $A\in\text{Cl}_X$ is such that, for each $i\in n$ and for each $B\in\text{Cl}_X$, if $B\cap(U_i\cup C)\subseteq X\setminus A$ and $A\cap(U_i\cup C)\notin\text{Kc}_X$, then $B\cap(U_i\cup C)\in\text{Kc}_X$. Let $K=\{p_i: i\in n\}$ be a set of cardinality $n$. Put $\text{Cl}_K=\mathcal{P}(K)$ and $\text{Cov}_K=\mathcal{P}^{2}(K)$. Define a mapping $\psi:\text{Cl}_K\to L(\text{Cl}_X)$ as follows: for $A\in \mathcal{P}(K)$, let $\psi(A)$ be the collection of all $D\in\text{Cl}_X$ such that $D\cap(U_i\cup C)\notin\text{Kc}_X$ if and only if $p_i\in A$. Then $\psi$ is a lattice isomorphism of $\text{Cl}_K$ onto $L(\text{Cl}_X)$. The compactification ${\alpha}^s_{\psi} X$ constructed in Section 6 is a weakly Hausdorff strict compactification of $X$ such that ${\alpha}^{s}_{\psi} X\setminus X$ is of cardinality $n$.
\end{rem} 
 
\section{Continuous extensions of mappings}

An interesting problem is to find reasonable necessary and sufficient conditions for a mapping to be extendable to a strictly continuous mapping over a strict compactification. In this section, we offer some solutions to this problem mainly for compactifications of Wallman type.

For collections $\mathcal{U}$ and $\mathcal{V}$ of sets, we write $\mathcal{U}\preceq \mathcal{V}$ when every member of $\mathcal{U}$ is contained in a member of $\mathcal{V}$ and $\bigcup\mc{U}=\bigcup\mc{V}$. Let us recall the notion of weak continuity (cf. \cite{Pie1}) and introduce the following important concepts of w-continuity and of W-continuity:

\begin{defi} Let $X$ and $Y$ be gtses and let $f:X\to Y$. We say that:
\begin{enumerate}
\item[(i)] $f$ is \textbf{w-continuous} if, for each (essentially) finite family $\mc{V}\subseteq\text{Op}_Y$ such that $\mathcal{V}$ covers $Y$, there exists an (essentially) finite family $\mc{U}\subseteq\text{Op}_X$ such that $\mc{U}$ covers $X$ and $\mc{U}\preceq f^{-1}(\mc{V})$.
\item[(ii)] $f$ is \textbf{W-continuous} if, for each family $\mc{V}\in\text{Cov}_Y$ such that $\mc{V}$ covers $Y$, there exists $\mc{U}\in\text{Cov}_X$ such that $\mc{U}$ covers $X$ and $\mc{U}\preceq f^{-1}(\mc{V})$.
\item[(iii)] $f$ is \textbf{weakly continuous} if $f:X_{top}\to Y_{top}$ is continuous.
\end{enumerate} 
\end{defi}

\begin{f}
Every strictly continuous mapping is simultaneously W-continuous and w-continuous.
\end{f}
\begin{f} Let $f:X\to Y$ be a mapping of a gts $X$ into a gts $Y$. Then the following properties hold:
\begin{enumerate}
\item[(i)] if $X$ is admissibly compact and if $f$ is W-continuous, then $f$ is w-continuous;
\item[(ii)] if $Y$ is admissibly compact and if $f$ is w-continuous, then $f$ is W-continuous;
\item[(iii)] when both $X$ and $Y$ are admissibly compact, then $f$ is W-continuous if and only if it is w-continuous.
\end{enumerate}
\end{f}
\begin{cor}
In the class of mappings between small gtses, the notions of W-continuity and w-continuity are equivalent. 
\end{cor}
\begin{exam} Let $X=Y=\omega, \text{Cov}_X=\text{EssFin}(\mc{P}({X}))$ and $\text{Cov}_Y=\mc{P}^2(Y)$. Then the identity $id_{\omega}$ as a mapping from the gts $(X, \text{Cov}_X)$ to the gts $(Y, \text{Cov}_Y)$ is w-continuous but not W-continuous.
\end{exam}

\begin{defi} Let us call a gts $Y$ zero-dimensional if, for each finite covering $\mc{V}\subseteq\text{Op}_Y$ of $Y$, there exists a pairwise disjoint finite covering $\mc{W}\subseteq\text{Op}_Y$ of $Y$ such that $\mc{W}\preceq \mc{V}$.
\end{defi}
\begin{prop}
Suppose that $f:X\to Y$ is a W-continuous mapping of a gts $X$ to a zero-dimensional gts $Y$. Then $f$ is w-continuous.
\end{prop}
\begin{proof} 
Let $\mc{V}\subseteq\text{Op}_Y$ be a finite covering of $Y$. Take a finite covering $\mc{W}\subseteq\text{Op}_Y$ of $Y$ such that $\mc{W}\preceq \mc{V}$ and $V\cap W=\emptyset$ for each pair $V, W$ of distinct members of $\mc{W}$. Since $f$ is W-continuous, there exists an admissible open covering $\mc{U}$ of $X$ such that $\mc{U}\preceq f^{-1}(\mc{W})$. For each $W\in\mc{W}$, let $\mc{U}[W]=\{ U\in\mc{U}: f(U)\subseteq W\}$ and let $U[W]=\bigcup\mc{U}[W]$.  Consider any $W\in\mc{W}$. To show that $U[W]\in\text{Op}_X$, let us observe that $U[W]\cap U= U$ if $U\in \mc{U}[W]$, while $U[W]\cap U=\emptyset$ if $U\in\mc{U}\setminus\mc{U}[W]$. Therefore, $U[W]\cap U\in\text{Op}_X$ for each $U\in\mc{U}$; thus, it follows from condition (A8) of Definition 2.2.1 of \cite{Pie1} that $U[W]=\bigcup \{ U[W]\cap U: U\in\mc{U}\}\in\text{Op}_X$. Then $\mc{U}_0=\{ U[W]: W\in\mc{W}\}\in\text{Cov}_X$ by condition (A3) of Definition 2.2.1 of \cite{Pie1}. Of course, $\mc{U}_0\preceq f^{-1}(\mc{V})$ which concludes the proof that $f$ is w-continuous.
\end{proof}

We regret that we cannot offer any example of a W-continuous but not w-continuous mapping.  

\begin{prop}
If $f:X\to Y$ is w-continuous or W-continuous and, simultaneously, if $\text{Cl}_Y$ is disjunctive, then $f$ is weakly continuous.
\end{prop}
\begin{proof}
For $x\in X$ and a weakly open $V$ in $Y$ such that  $f(x)\in V$, there exists $A\in Cl_Y$
such that $f(x)\in Y\setminus A\subseteq V$. When $\text{Cl}_Y$ is disjunctive, there exists $B\in Cl_Y$ such  that $f(x)\in B\subseteq Y\setminus A$. By the w-continuity  or the W-continuity of $f$, for $\mc{V}=\{Y\setminus A,Y\setminus B\}$, one can find $\mc{U}\in\text{Cov}_X$ such that $\mc{U}$ covers $X$ and $\mc{U}\preceq f^{-1}(\mc{V})$. There is $U_x\in \mc{U}$ such that $x\in U_x$. Then $f(U_x)\subseteq Y\setminus A\subseteq V$ because $f(x)\in f(U_x)\cap B$.  
\end{proof}
 \begin{defi} Let $X$ and $Y$ be weakly normal gtses. We say that a mapping $\hat{f}:\mathcal{W}(X,\text{Cl}_X)\to \mathcal{W}(Y, \text{Cl}_Y)$ is a Wallman extension of a mapping $f:X\to Y$ if $\hat{f}(x)=f(x)$ for each $x\in X$.
 \end{defi}
\begin{thm}
Let $X$ and $Y$ be weakly normal gtses such that the Wallman spaces $\mathcal{W}(X,\text{Cl}_X)$ and $\mathcal{W}(Y, \text{Cl}_Y)$ are compact. Then a mapping $f:X\to Y$ has a weakly continuous Wallman extension if and only if $f$ is w-continuous.
\end{thm}
\begin{proof}
\textsl{Necessity.}
Suppose $\hat{f}$ is a weakly continuous Wallman extension of $f$. For a finite covering $\mc{V}\subseteq\text{Op}_Y$ of $Y$, define $\bar{\mc{V}}=\{ \text{Ex}_{\mathcal{W}(Y, \text{Cl}_Y)}(V):V\in \mc{V}\}$. Then $\bar{\mc{V}}$ is weakly open and it covers $\mathcal{W}(Y, \text{Cl}_Y)$ because, by Theorem 5.16, the operator $\text{Ex}_{\mathcal{W}(Y, \text{Cl}_Y)}$ is finitely additive. Let $\mathcal{U}$ be the collection of all sets $U\in\text{Op}^{w}_{\mathcal{W}(X,\text{Cl}_X)}$ such that $\hat{f}(U)$ is contained in a member of $\bar{\mc{V}}$. Since $\hat{f}$ is weakly continuous, $\mathcal{U}$ is an open cover of $\mathcal{W}(X,\text{Cl}_X)$. By the compactness of $\mathcal{W}(X, \text{Cl}_X)$, there exists a finite $\mathcal{U}_0\subseteq\mathcal{U}$ such that $\mathcal{U}_0$ covers $X$. Then the collection $X\cap_1\mathcal{U}_0\in\text{Cov}_X$ covers $X$ and $X\cap_1\mathcal{U}_0\preceq f^{-1}(\mathcal{V})$, so $f$ is w-continuous.

\textsl{Sufficiency.} Assume $f$ is w-continuous. Then, in view of Proposition 8.8, $f:X\to Y$ is weakly continuous. By the compactness of $\mathcal{W}(Y,\text{Cl}_Y)$,  for a pair $A_1, A_2$ of  disjoint sets that are closed in the Wallman space $\mathcal{W}(Y, \text{Cl}_Y)$,  there exists a pair $D_1,D_2$ of disjoint members of $\text{Cl}_Y$ such that $A_i\subseteq [D_i]_{\text{Cl}_Y}$ for $i\in\{1, 2\}$. It follows from the
w-continuity of $f$ that the finite covering $\mc{W}=\{Y\setminus D_1, Y\setminus D_2\}\subseteq\text{Op}_Y$ of $Y$
admits a finite covering $\mc{G}\subseteq\text{Op}_X$ of $X$ such that $\mc{G}\preceq f^{-1}(\mc{W})$. Let $\mc{G}_i=\{ G\in \mc{G}:f(G)\subseteq Y\setminus D_i\}$ and $B_i=X\setminus \bigcup \mc{G}_i$ for $i\in\{1, 2\}$. Since the sets $B_i\supseteq f^{-1}(A_i)$ are disjoint and they belong to $\text{Cl}_X$, the sets $[B_i]_{\text{Cl}_X}$ are disjoint. Hence  $ \cl_{\mc{W}(X, \text{Cl}_X)} f^{-1}(A_1)\cap\cl_{\mc{W}(X,\text{Cl}_X)} f^{-1}(A_2)=\vn$. In the light of Theorem 5.15, the mapping $f:X\to \mc{W}(Y, \text{Cl}_Y)$ is extendable to a weakly continuous mapping $\hat{f}:\mathcal{W}(X,\text{Cl}_X)\to \mathcal{W}(Y, \text{Cl}_Y)$.  
\end{proof}

\begin{rem} Let $\mathcal{C}$ be a Wallman base of a semi-normal space $X$  and let $Y$ be a compact Hausdorff topological space. In \cite{Fr}, O. Frink called a mapping $f:X\to Y$ uniformly continuous relative to $\mc{C}$ and $Y$ if, for every finite open cover $\mc{V}$ of $Y$, there exists a finite cover $\mc{U}\subseteq\{ X\setminus C: C\in\mc{C}\}$ of $X$, such that $\mc{U}\preceq f^{-1}(\mc{V})$. Using \textbf{AC}, he proved the following extension theorem: a mapping $f:X\to Y$ is continuously extendable to a mapping $\hat{f}:\mc{W}(X, \mc{C})\to Y$ if and only if $f$ is uniformly continuous relative to $\mc{C}$ and $Y$. One can see that our notion of w-continuity can be regarded as a generalization of Frink's notion of uniform continuity relative to $\mc{C}$ and $Y$. Frink's extension theorem can be deduced from our Theorem 8.10. Moreover, in view of our Theorems 2.8 and 8.10, Frink's extension theorem is valid in \textbf{ZFU}. By Theorems 2.8 and 8.10,  also Theorem 2 of \cite{Fr}, although proved by Frink only in \textbf{ZFC}, is valid in \textbf{ZFU}.
\end{rem}
 
\begin{prop}
(\textbf{ZFU}) Every strictly continuous mapping of a weakly normal gts to a weakly normal gts has a weakly continuous Wallman extension.
\end{prop}
\begin{proof} It suffices to apply Theorems 2.8 and 8.10 together with Fact 8.2.
\end{proof}

\begin{prop}
Let  $\alpha X$ and $\gamma Y$ be strict compactifications of gtses $X$ and $Y$, respectively. Suppose that $f: X\to Y$ is a strictly continuous mapping which has a weakly continuous extension $\hat{f}:\alpha X\to\gamma Y$. If the generalized topologies of $\alpha X$ and of $\gamma Y$ are the strongest generalized topologies associated with $\text{Cov}_X$ and $\text{Cov}_Y$, respectively, then $\hat{f}$ is strictly continuous.
\end{prop} 
\begin{proof}  Let $\mc{V}\in\text{Cov}^{S}_{\gamma Y}$. It follows from the weak continuity of $\hat{f}$ that $\hat{f}^{-1}(V)$ is weakly open in $\alpha X$ whenever $V\in\mc{V}$. Since $\mc{V}\cap_1 Y\in\text{Cov}_Y$, we have $\hat{f}^{-1}(\mc{V})\cap_1 X\in\text{Cov}_X$ because $f$ is strictly continuous. This implies that $\hat{f}^{-1}(\text{Cov}^{S}_{\gamma Y})\subseteq \text{Cov}^{S}_{\alpha X}$. 
\end{proof}

From Theorem 8.10 and Proposition 8.13, we immediately deduce the following:

\begin{thm} 
Suppose that $X_1$ and $X_2$ are weakly normal gtses such that the Wallman spaces $\mc{W}(X_i, \text{Cl}_{X_i})$ are compact for $i\in\{0, 1\}$. Let $\text{COV}_{i}$ be the strongest generalized topology in $\mc{W}(X_i,\text{Cl}_{X_i})$ associated with $\text{Cov}_{X_i}$ for $i\in\{0, 1\}$. Then every strictly continuous mapping $f:X_1\to X_2$ has a strictly continuous extension $\hat{f}: (\mc{W}(X_1, \text{Cl}_{X_1}), \text{COV}_1)\to (\mc{W}(X_2, \text{Cl}_{X_2}), \text{COV}_2)$. 
\end{thm}
 
As usual, for a topological space $X$,  we denote by $C(X)$ the ring of all continuous functions from $X$ into the real line $\mathbb{R}$ equipped with its natural topology. For a Hausdorff compactification $\alpha X$ of a topological space $X$, let $C_{\alpha}(X)$ be the ring of all functions $f\in C(X)$ that are continuously extendable over $\alpha X$ and let $\mathcal{Z}_{\alpha}(X)=\{ f^{-1}(0): f\in C_{\alpha}(X)\}$. The sets $C(X)$ and $C_{\alpha}(X)$ will be introduced also for gtses $X$ in Definition 10.7.

\begin{thm}
(\textbf{ZFC}) Suppose that $X$ and $Y$ are weakly normal small gtses such that there exist Hausdorff compactifications $\alpha X$ of $X_{top}$ and  $\gamma Y$ of $Y_{top}$ with the following properties: $\text{Cl}_X=\mc{Z}_{\alpha}(X)$ and $\text{Cl}_Y=\mc{Z}_{\gamma}(Y)$. Then a mapping $f: X\to Y$ has a weakly continuous Wallman extension if and only if $f$ is strictly continuous. 
\end{thm}
\begin{proof}
In view of Theorem 8.10 and Proposition 8.12, it suffices to prove that if $f$ is w-continuous and $Y$ is small, then $f$ is strictly continuous. Therefore, assume that $f$ is w-continuous and that $\hat{f}$ is its weakly continuous Wallman extension.  Let $A\in\text{Cl}_Y$. Since, by Theorem 2.2 of \cite{W1}, $\gamma Y\leq \mc{W}(Y, \text{Cl}_Y)$, there exists $h\in C(\mc{W}(Y, \text{Cl}_Y))$ such that $A=Y\cap h^{-1}(0)$. Then $g=h\circ\hat{f}\in C(\mc{W}(X, \text{Cl}_X))$ and $f^{-1}(A)=X\cap g^{-1}(0)$.  This,  together with Theorem 2.9 of \cite{W1}, shows that $f^{-1}(A)\in\text{Cl}_X$. Hence $f^{-1}(\text{Cl}_Y)\subseteq \text{Cl}_X$ and, in consequence. $f$ is strictly continuous.   
\end{proof}

That not every w-continuous mapping between weakly normal small gtses is strictly continuous is shown by the following example:

\begin{exam} We consider the real line $\mathbb{R}$ equipped with its natural topology. Let $\mc{C}$ be the collection of all the closed sets in $\mathbb{R}$ that are either compact or their complements do not contain non-compact closed sets in $\mathbb{R}$. Let $\mc{D}$ be the collection of all sets $D\in\mc{C}$ that $D$ has a finite number of connected components. Then both $\mc{C}$ and $\mc{D}$ are Wallman bases for $\mathbb{R}$ such that $\mc{C}\neq\mc{D}$, while the Wallman spaces $\mc{W}(\mathbb{R}, \mc{C})$ and $\mc{W}(\mathbb{R},\mc{D})$ are topologically equivalent one-point compactifications of $\mathbb{R}$. Therefore, in view of Theorem 8.10, the mapping $\text{id}_{\mathbb{R}}$ is a w-continuous but not strictly continuous mapping of the small gts $(\mathbb{R},\mathcal{D})$ into the small gts $(\mathbb{R}, \mathcal{C})$ (cf. Remark 2.2).
\end{exam}

 \section{$I_g$-Tychonoff spaces}
 
We are going to apply embeddings into products to our theory of strict compactifications. To do this well, we need a notion of a completely regular gts. Of course, we have already defined topologically completely regular gtses (cf. Definition 3.1); however, we want to equip $\mathbb{R}$ with a generalized topology $\text{Cov}$ to use strictly continuous mappings of gtses into $(\mathbb{R}, \text{Cov})$ in our concept of a Tychonoff gts. Among many generalized topologies in $\mathbb{R}$ that induce the natural topology $\tau_{nat}$ of $\mathbb{R}$, there are all generalized topologies described in Definition 1.2. It is good to learn how many distinct weakly normal generalized topologies inducing the discrete topology of a given infinite set $X$ there are in $X$.  
 
\begin{defi} 
Let $\mathbf{ON}$ be the class of all ordinal numbers of von Neumann and let $X$ be a set. Then:
\begin{enumerate}
\item[(i)] we say that $X$ has its cardinality if there exists $\alpha\in\mathbf{ON}$ equipollent with $X$;
\item[(ii)] if $X$ has its cardinality, then the initial ordinal number $\mid X\mid \in\mathbf{ON}$ equipollent with $X$ is called the cardinality of $X$.
\end{enumerate}
\end{defi}

\begin{rem} It is valid in \textbf{ZFC} that every set has its cardinality. In every model for \textbf{ZF+$\neg$AC($\mathbb{R}$)}, the set $\mathbb{R}$ does not have its cardinality.
\end{rem}

\begin{thm}
(\textbf{ZFC}) Let $X$ be an infinite set of cardinality $\kappa$. Then there are exactly $2^{2^{\kappa}}$ weakly normal small generalized topologies in $X$ that induce the discrete topology of $X$.
\end{thm}
\begin{proof} Let us consider $X$ with its discrete topology. We denote by $\mc{WNS}(X)$  the set of all weakly normal small generalized topologies that induce the discrete topology of $X$. Let $\mathcal{K}$ be the collection of all two-point subsets for $\beta X\setminus X$ where $\beta X$ stands for the \v{C}ech-Stone compactification of the discrete space $X$. For each $K\in\mathcal{K}$, let 
$$\mc{C}_K=\{ A\in\mathcal{P}(X): K\cap\text{cl}_{\beta X}A=\emptyset \vee  K\subseteq\text{cl}_{\beta X}A \}.$$ 
Then $\mc{C}_K$ is a Wallman base of $X$ such that the compactification $\mc{W}(X, \mc{C}_K)$ is topologically equivalent with the compactification of $X$ obtained from $\beta X$ by identifying the set $K$ with a point. The generalized topologies $\text{Cov}(\mc{C}_K)=\text{EssFin}(\{ X\setminus A: A\in\mc{C}_K\})$ with $K\in\mc{K}$ are all weakly normal, they induce the discrete topology in $X$ and they are pairwise distinct. By Theorem 3.6.11 of \cite{En}, the collection $\{\text{Cov}(\mc{C}_K): K\in\mc{K}\}$ is of cardinality $2^{2^{\kappa}}$; hence $2^{2^{\kappa}}\leq \mid\mc{WNS}(X)\mid$. On the other hand, since the set of all Wallman bases of the discrete space $X$ is a subset of $\mc{P}^2(X)$ and it is equipollent with  $\mc{WNS}(X)$, we have  $\mid\mc{WNS}(X)\mid\leq 2^{2^{\kappa}}$.
\end{proof}
 
\begin{cor}
(\textbf{ZFC}) There are exactly  $2^{2^{\omega}}$ weakly normal small generalized topologies in $\mathbb{R}$ that induce the natural topology of $\mathbb{R}$.
\end{cor}
\begin{proof}
Let $\mc{WN}(\mathbb{R})$ be the set of all Wallman bases for $\mathbb{R}$ with the natural topology.  We assume that $\omega$ is the set of all non-negative integers of $\mathbb{R}$. Since $\beta\omega=\text{cl}_{\beta\mathbb{R}}\omega$, we infer from Theorem 9.3 that the cardinality of $\mc{WN}(\mathbb{R})$ is at least $2^{2^{\omega}}$. Since the collection of all closed sets of $\mathbb{R}$ is of cardinality $2^{\omega}$, the cardinality of $\mc{WN}(\mathbb{R})$ is at most $2^{2^{\omega}}$ which completes the proof.
\end{proof}

Let us establish  several simple facts about the generalized topologies from Definition 1.2 to make it clear that they are pairwise non-isomorphic, i.e. not strictly homeomorphic.

\begin{f}
\begin{enumerate}
\item[(i)] The line $\mb{R}_{ut}$ is not locally small.
\item[(ii)] The real lines $\mb{R}_{lom}$, $\mb{R}_{lst}$,$\mb{R}_{l^+om}$, $\mb{R}_{l^+st}$ are locally small, but not small. 
\item[(iii)] Each open set in small $\mb{R}_{om}$ has a finite number of connected components, which is false in small spaces $\mb{R}_{st}$, $\mb{R}_{slom}$, $\mb{R}_{sl^+om}$. 
\item[(iv)] Each bounded open interval of $\mb{R}_{lom}$, $\mb{R}_{l^+om}$ $\mb{R}_{slom}$ or $\mb{R}_{sl^+om}$ is isomorphic to $\mb{R}_{om}$, which is false for $\mb{R}_{st}$, $\mb{R}_{lst}$ and $\mb{R}_{l^+st}$. 
\item[(v)]Extracting one point from $\mb{R}_{l^+om}$ (or $\mb{R}_{sl^+om}$, resp.) disconnects the line into two pieces: a piece isomorphic to $\mb{R}_{om}$ and a piece isomorphic to $\mb{R}_{lom}$ (or $\mb{R}_{slom}$, resp.). 
\item[(vi)] Extracting one point from $\mb{R}_{om}$,$\mb{R}_{lom}$,$\mb{R}_{st}$, $\mb{R}_{lst}$, $\mb{R}_{slom}$  disconnects the line into two mutually isomorphic pieces. 
\item[(vii)] None of the real lines from Definition 1.2 (i)-(ix) is strictly homeomorphic with $\mb{R}_{rom}$ because $\mb{R}_{rom}$ has only countably many open sets.
\end{enumerate}
\end{f}
\begin{cor}
The lines from Definition 1.2 (i)--(x)  are pairwise non-isomorphic when, among the lines localized at $+\infty$ or $-\infty$, only the ones localized at $+\infty$ are taken into consideration. 
\end{cor}

For real numbers $a, b$ such that $a<b$, let $\tau_{nat}^{[a, b]}$ be the natural topology in $[a, b]$ and let $$\mc{G}_{\mathbb{R}}=\{ ut, om, st, lom, lst, slom, l^{-}om, l^{+}om, l^{-}st, l^{+}st, sl^{-}om, sl^{+}om, rom\}.$$ For $g\in \mc{G}_{\mathbb{R}}$, the set $[a, b]$ is strict in the gts $\mathbb{R}_g$ and we denote by $[a, b]_g$ the interval $[a, b]$ equipped with the generalized topology of a subspace of the gts ${\mathbb{R}}_{g}$. It is easily seen that $\tau_{nat}^{[a, b]}$ is induced by the generalized topology of $[a, b]_{g}$ and that the gts $[a, b]_{g}$ is weakly normal whenever $g\in \mc{G}_{\mathbb{R}}$.  Moreover,  $\{ [a, b]_g: g\in \mc{G}_{\mathbb{R}}\}= \{[a, b]_{ut}, [a, b]_{om}, [a, b]_{st}, [a, b]_{rom}\}$.

\begin{defi}\label{odcinki}
Let $I=[0, 1]$ be the unit interval of $\mathbb{R}$. Then:
\begin{enumerate}
\item[(i)] the gts $I_{ut}$ will be called the \textbf{topological unit interval};
\item[(ii)] the gts $I_{om}$ will be called the \textbf{o-minimal unit interval};
\item[(iii)] the gts $I_{st}$ will be called the \textbf{smallified topological unit interval};
\item[(iv)] the gts $I_{rom}$ will be called the \textbf{rationalized o-minimal unit interval}.
\end{enumerate}
\end{defi}

\begin{f} 
\begin{enumerate}
\item[(i)] The interval $I_{ut}$ is not even locally small. 
\item[(ii)] The intervals $I_{om}$, $I_{st}$ and $I_{rom}$ are small.
\item[(iii)]  Each open set in $I_{om}$ has a finite number of connected components and this is false for $I_{st}$.
\item[(iv)]  The collection of all open sets of $I_{om}$ is not equipollent with the collection of all open sets of $I_{rom}$. 
\end{enumerate}
\end{f}
\begin{cor}
The intervals $I_{ut}, I_{om}, I_{st}, I_{rom}$ are pairwise non-isomorphic.
\end{cor} 

In what follows, let $I_g$ be one of the real unit intervals considered in Definition \ref{odcinki} where $g\in \mc{G}_{I}=\{ ut, om, st, rom\}$.

\begin{defi}
Subsets $A,B$ of a gts $X$ are $I_g$-\textbf{functionally separable} if there exists a strictly continuous mapping $f:X\to I_g$ such that $A\subseteq f^{-1}(0)$ and $B\subseteq f^{-1}(1)$. 
\end{defi}

\begin{defi}
A gts $X$ is $I_g$-\textbf{completely regular} if, for each $A\in\text{Cl}_X$ and for each $x\in X\setminus A$, the sets $A$ and $\{ x\}$ are $I_g$-functionally separable.
\end{defi}

\begin{defi}
A gts $X$ is $I_g$-\textbf{Tychonoff} if it is simultaneously weakly $T_1$ and $I_g$-completely regular.
\end{defi}

\begin{prop}
\begin{enumerate}
\item[(i)]Every $I_{ut}$-completely regular gts is $I_{st}$-completely regular.
\item[(ii)] Every $I_{st}$-completely regular gts is $I_{om}$-completely regular. 
\item[(iii)] Every $I_{om}$-completely regular space is $I_{rom}$-completely regular.
\item[(iv)] A topological gts $X$ is $I_{ut}$-completely regular if and only if $X_{top}$ is completely regular.
\end{enumerate}
\end{prop}
\begin{proof} For an arbitrary gts $X$, it suffices to observe that
each strictly continuous mapping $f:X\to I_{ut}$ is also strictly continuous if considered as $f:X\to I_{st}$ and, moreover, each strictly continuous mapping $f:X\to I_{st}$ is also strictly continuous if considered as $f:X\to I_{om}$. 
\end{proof}

\begin{exam}
It is easily seen that the real lines from Definition 1.2 have the following properties:
\begin{enumerate}
\item[(i)] the topological real line $\mb{R}_{ut}$ is $I_{ut}$-Tychonoff;
\item[(ii)] the lines $\mb{R}_{st}$, $\mb{R}_{lst}$,  $\mb{R}_{l^+st}$ ($\mb{R}_{l^-st}$) are $I_{st}$-Tychonoff but not $I_{ut}$-Tychonoff;
\item[ (iii)] all of the nine real lines of Definition 1.2 (i)-(ix) are $I_{om}$-Tychonoff;
\item[(iv)] the rationalized o-minimal real line $\mb{R}_{rom}$ is not $I_{om}$-Tychonoff;
\item[(v)] the rationalized o-minimal unit interval $I_{rom}$ is an example of an absolutely compact weakly normal gts which is not $I_{om}$-Tychonoff.
\end{enumerate}
\end{exam}

\begin{q} 
For a set $X\subseteq \mb{R}$, let $\mc{NG}(X)$ be the collection of all weakly normal generalized topologies on $X$ that induce the natural topology in $X$. Is it true that, for every $\text{Cov}^{\mb{R}}\in\mc{NG}(\mb{R})$, there exists $\text{Cov}^I\in\mc{NG}(I)$ such that the gts $(\mb{R}, \text{Cov}^{\mb{R}})$ is not $(I, \text{Cov}^{I})$-Tychonoff?
\end{q}

\section{Embeddings into products}

In the classical theory of Hausdorff compactifications in \textbf{ZFC}, an important role is played by evaluation mappings into Tychonoff cubes (cf. e. g. \cite{BY1}-\cite{BY3}, \cite{Bl}, \cite{W1}-\cite{W3}). Let us recall that, for an indexed set $F=\{f_j: j\in J\}$ of mappings $f_j: X\to Y_j$, the evaluation mapping $e_F: X\to\prod_{j\in J}Y_j$ is defined by: $[e_F (x)](j)=f_j(x)$ for $x\in X$ and $j\in J$. If $X$ is a topological space, $\{Y_j: j\in J\}$ is a collection of topological spaces and $F=\{ f_j: j\in J\}$ where each $f_j$ is a mapping from $X$ into $Y_j$, there are nice necessary and sufficient conditions for $e_{F}$ to be a homeomorphic embedding (cf. e. g. \cite{Ch},  \cite{BY1}-\cite{BY3}, \cite{W1}-\cite{W3} and 2.3.D of \cite{En}); however, the following interesting problem is new and unsolved:

\begin{p}
Let $F=\{f_j: j\in J\}$ be a collection of mappings $f_j:X\to Y_{j}$ where $X$ and $Y_j$ are gtses for $j\in J$. Find elegant necessary and sufficient conditions for $e_{F}$ to be a strict embedding of $X$ into the $\mathbf{GTS}$-product of the family $\{ Y_j: j\in J\}$.
\end{p}

\begin{f} If $F=\{ f_j: j\in J\}$ is a collection of strictly continuous mappings $f_j: X\to Y_j$ where $X$ and $Y_j$ are gtses, then $e_F$ is a strictly continuous mapping of $X$ into $\prod_{j\in J}^{\mathbf{GTS}}Y_j$.
\end{f}

The following facts may be inferred from Proposition 2.2.37 of \cite{Pie1}.
\begin{f}
For $\Psi\in \mc{P}^3(X)$ and $A\subseteq X$ we have $\langle\langle\Psi\rangle_X \cap_2 A\rangle_A=\langle\Psi \cap_2 A\rangle_A$.
\end{f}

\begin{f}
For sets $X,Y$ and $\Psi\in \mc{P}^3(X)$ we have $\langle\langle\Psi\rangle_X \times_2 Y\rangle_{X\times Y}=
\langle\Psi\rangle_X \times_2 Y
=\langle\Psi \times_2 Y\rangle_{X\times Y}$.
\end{f}

\begin{lem}\label{produkty}
Let $\{ X_j: j\in J\}$ be a collection of gtses and, for each $j\in J$, let $A_j$ be a subspace of $X_j$. Then $\prod^{\mathbf{GTS}}_{j\in J} A_j$ is a subspace of $\prod^{\mathbf{GTS}}_{j\in J} X_j$.
\end{lem}

\begin{proof}
We may assume that $J$ has at least two elements. If $Y_j\subseteq X_j$ for all $j\in J$, then, for a fixed $j\in J$, let $Y^{(j)}=\prod_{i\in J\setminus\{j\}}Y_i$.  Put $A=\prod_{j\in J}A_j$.  Let $\tau_1$  be the $\mathbf{GTS}$-product topology on $A$. Then
$\tau_1=\langle \bigcup_{j\in J}( \text{Cov}_{A_j}\times_2 A^{(j)})\rangle_A$. 

Let $\tau_2$ be the induced  from $\prod^{\mathbf{GTS}}_{j\in J} X_j$ generalized topology  on $A$. Then $\tau_2=
\langle\langle \bigcup_{j\in J}(\text{Cov}_{X_j}\times_2 X^{(j)})\rangle_X\cap_2 A\rangle_A$. To see that $\tau_1=\tau_2$, it suffices to apply Facts 10.3 and 10.4 
\end{proof}

The following proposition is a partial solution to Problem 10.1:

\begin{prop}
Let $X$ be a gts and let $\{Y_j: j\in J\}$ be an indexed set of gtses $Y_j$, Suppose that $f_j:X\to Y_j$ is strictly continuous for each $j\in J$ and that there exists $j_0\in J$ such that $f_{j_0}$ is an embedding. Then, for $F=\{ f_j: j\in J\}$, the evaluation mapping $e_F:X\to\prod^{\mathbf{GTS}}_{j\in J}Y_j$ is an embedding.
\end{prop}
\begin{proof}
Let $Cov_F$ be the $\mathbf{GTS}$-product generalized topology in $\prod_{i\in J}f_j(X)$ and let $Cov$ be the generalized topology of $\prod^{\mathbf{GTS}}_{j\in J}Y_j$. We deduce from Fact 10.2 that $e_F$ is strictly continuous. We need to show that $e_F(\text{Cov}_X)\subseteq\langle e_F(X){\cap}_2 Cov\rangle_{e_F(X)}$. Assume that $J$ has at least two elements. Let $\mc{U}\in\text{Cov}_X$. Observe that $e_F(\mc{U})= e_F(X)\cap_1 [f_{j_0}(\mc{U})\times_1\prod_{j\in J\setminus\{ j_0\}}f_j(X)]\in  e_F(X)\cap_2 Cov_F$. Since, in view of Lemma 10.5, $\prod^{\mathbf{GTS}}_{j\in J}f_j(X)$ is a subspace of $\prod^{\mathbf{GTS}}_{j\in J}Y_j$, we obtain that $e_F(\mc{U})\in\langle e_F(X)\cap_2 Cov\rangle_{e_F(X)}$.
\end{proof}

In what follows, $\mb{R}$ stands for the real line with $\tau_{nat}$. If $X$ is a topological space, then let $X_{top}=X$.

\begin{defi}
Suppose that $g\in \mc{G}_{\mb{R}}$ and $\alpha X$ is a strict compactification of a gts $X$. Then:
\begin{enumerate}
\item[(i)] $C(X)=C(X_{top})$ is the collection of all weakly continuous functions $f:X\to\mb{R}$;
\item[(ii)] $C^{\ast}(X)=C^{\ast}(X_{top})$ is the collection of all bounded functions from $C(X)$;
\item[(iii)] $C^{g}(X)$ is the collection of all strictly continuous functions $f:X\to\mb{R}_g$;
\item[(iv)] $C^{\ast g}(X)=C^{\ast}(X)\cap C^{g}(X)$;
\item[(v)] $C_{\alpha}(X)$ is the collection of all functions from $C(X)$ that are continuously extendable over $(\alpha X)_{top}$;
\item[(vi)] $C^{g}_{\alpha}(X)$ is the collection of all functions from $C^g(X)$ that are extendable to functions from $C^{g}(\alpha X)$;
\item[(vii)] if $f\in C_{\alpha}( X)$, let $f^{\alpha}\in C(\alpha X)$ be such that $f^{\alpha}(x)=f(x)$ for each $x\in X$;
\item[(viii)] if $F\subseteq C_{\alpha}(X)$, let $F^{\alpha}=\{ f^{\alpha}: f\in\mc{F}\}$;
\item[(ix)] $\mc{E}(X)=\mc{E}(X_{top})$ is the collection of all sets $F\subseteq C^{\ast}(X)$ such that $e_F$ is a homeomorphic embedding;
\item[(x)] $\mc{E}^g(X)$ is the collection of all sets $F\subseteq C^{\ast g}(X)$ such that $e_F$ is a strict embedding of $X$ into the $\mathbf{GTS}$- product $\mb{R}_g^{F}$.
\item[(xi)] if $F\in\mc{E}(X)$, then $e_F X$ is the weak closure of $e_F(X)$ in the Tychonoff product $\mb{R}^{F}$.
\end{enumerate}
\end{defi}
\begin{f}
Let $\mc{G}^{\ast}_\mb{R}=\{ ut, om, st, rom\}$. Then, for every gts $X$, we have $\{ C^{\ast g}(X): g\in \mc{G}_{\mb{R}}\}=\{ C^{\ast g}(X): g\in \mc{G}^{\ast}_{\mb{R}}\}$.
\end{f}

\begin{rem}
Let us warn that since addition and multiplication in small real lines are not strictly continuous, the sets $C^{g}(X)$ need not be 
algebras with respect to the standard addition and multiplication of functions. For example, $C^{st}(\mb{R}_{st}\times_{\mathbf{GTS}}\mb{R}_{st})$ and $C^{om}(\mb{R}_{om}\times_{\mathbf{GTS}}\mb{R}_{om})$ are not algebras because if $f(x, y)=x $ and $g(x,y)=y$ for $(x, y)\in\mb{R}^2$, then neither $f+g$ nor $f\cdot g$ belongs to $C^{st}(\mb{R}_{st}\times_{\mathbf{GTS}}\mb{R}_{st})\cup C^{om}(\mb{R}_{om}\times_{\mathbf{GTS}} \mb{R}_{om})$. Of course, $C^{st}(\mb{R}_{st})=C(\mb{R}_{st})$ is an algebra. It is interesting that $C^{om}(\mb{R}_{om})$ is not an algebra because, for instance, if $f_1(x)=2x+ \sin x$ and $f_2(x)=-2x$ for $x\in\mb{R}$, then $f_1, f_2\in C^{om}(\mb{R}_{om})$ but $f_1+f_2\notin C^{om}(\mb{R}_{om})$. 
\end{rem}

Since $\mathbf{GTS}$-products of small gtses are small and every subspace of a small space is small, we can observe the following fact:
\begin{f}
If $g\in\{ om, st, rom\}$, then, for every $I^{g}$-Tychonoff not small gts $X$, the collection $\mc{E}^{g}(X)$ is empty. 
\end{f}
\begin{exam} 
The space $\mb{R}_{ut}$ is $I_{st}$-Tychonoff but there does not exit a subset $F$ of $C^{st}(\mb{R}_{ut})$ such that $e_{F}$ is a strict embedding of $\mb{R}_{ut}$ into the $\mathbf{GTS}$-product $\mb{R}_{st}^{F}$.
\end{exam}

Since we work in \textbf{ZF}, contrary to the results of \cite{Ch},  \cite{BY1}-\cite{BY3} and \cite{W1}-\cite{W3} in \textbf{ZFC}, we should not claim that if $X$ is a topological space or a gts and if $F\in\mc{E}(X)$ , then $e_F X$ is certainly topologically compact. 

\begin{thm}  Equivalent are:
\begin{enumerate}
\item[(i)] \textbf{UFT};
\item[(ii)] for every Tychonoff space $X$ and for each $F\in\mc{E}(X)$, the space $e_F X$ is topologically compact;
\item[(iii)] for every Tychonoff space $X$ and for each $F\in \mc{E}(X)$, there exists a topological compactification $\alpha X$ of $X$ such that $F\subseteq C_{\alpha}(X)$.
\end{enumerate}
\end{thm}
\begin{proof} To show that (i) implies (ii), it is sufficient to apply Theorem 4.70 of \cite{Her}. We put $\alpha X=e_F X$ to check that (iii) follows from (ii). Let us assume that (iii) holds. Take any non-void set $J$ and put $I_j=[0, 1]$ with its natural topology for each $j\in J$. Let $X=I^{J}$ be the Tychonoff product $\prod_{j\in J}I_j$. For $j\in J$, let $\pi_j: X\to J$ be the projection. Consider the set $F=\{ \pi_j: j\in j\}$. Clearly, $F\in \mc{E}(X)$. Let $\alpha X$ be a compactification of $X$ such that $F\subseteq C_{\alpha}(X)$. Then the mapping $h=e_{F^{\alpha}}$ is weakly continuous,  while $h(\alpha X)=X$. Hence $X$ is topologically compact. This, together with Theorem 4.70 of \cite{Her}, proves that (iii) implies \textbf{UFT}.
\end{proof}

\begin{exam} Let \textbf{M} be any model for $\mathbf{ZF +\neg UFT}$, for example, let \textbf{M} be the Pincus-Solovay's model $M27$ of \cite{HR} (cf. also Appendix A6 of \cite{Her}). In \textbf{M}, there is a non-compact Tychonoff cube. Let $I^{J}$ be a Tychonoff cube which is non-compact in \textbf{M}. Then there does not exist a compactification $\alpha I^{J}$ such that $C^{\ast}(I^J)\subseteq C_{\alpha}(I^J)$ in \textbf{M}; moreover, for $F=C^{\ast}(I^J)$, the pair $(e_F I^J, e_F)$ is not a compactification of $I^{J}$ in \textbf{M}.
\end{exam} 

\begin{rem} Let $g\in\mc{G}^{\ast}_{\mb{R}}$ and let $X$ be an $I^{g}$-Tychonoff gts. Suppose that $F\in\mc{E}(X_{top})$. Assume \textbf{ZFU}. Even when we do not know whether $F\in\mc{E}^g(X)$, we can equip the topological compactification $e_F X$ of $X$ with the strongest generalized topology associated with $\text{Cov}_X$ to obtain a strict compactification $e^S_F X$ of $X$. 
\end{rem}

Let us finish with the following open problem:
\begin{p}
Let $g\in \mc{G}^{\ast}_{\mb{R}}$ and let $X$ be an $I^g$-Tychonoff gts. Suppose that $\alpha X$ is a strict compactification of $X$ and that $F\in\mc{E}(X_{top})$. Assume \textbf{ZFU} and let $e^g_F X$ be the set $e_F X$ equipped with the generalized topology of the subspace of the $\mathbf{GTS}$-product $\mb{R}_g^F$. Find elegant necessary and sufficient conditions for $e^g_F X$ to be a strict compactification of $X$, strictly equivalent with $\alpha X$. 
\end{p}

\end{document}